\documentclass{amsart}

\usepackage{setspace}

\usepackage{amsrefs}
\usepackage[nospace,noadjust]{cite}
\usepackage[dvipsnames]{xcolor}
\usepackage{graphicx}
\usepackage{caption}
\usepackage{mathrsfs}  
\usepackage{subcaption}
\usepackage{amsmath} 
\usepackage{amssymb}
\usepackage{listings}
\usepackage{booktabs}
\usepackage{tikz}
\usepackage{lipsum}
\usepackage{mdwlist}

\newtheorem{theorem}{Theorem}[section]
\newtheorem{lemma}[theorem]{Lemma}
\newtheorem{corollary}[theorem]{Corollary}
\newtheorem{proposition}[theorem]{Proposition}

\newtheorem{problem}{Problem}

\theoremstyle{definition}
	\newtheorem{definition}[theorem]{Definition}
	\newtheorem{example}[theorem]{Example}

\theoremstyle{remark}
	\newtheorem{remark}[theorem]{Remark}

\numberwithin{equation}{section}

\renewcommand{\Re}{\operatorname{Re}}

\copyrightinfo{2020}{Michel L. Lapidus, Machiel van Frankenhuijsen, Edward K. Voskanian}

\begin{document}

\title[]{Quasiperiodic Patterns of the Complex Dimensions of Nonlattice
Self-Similar Strings, \\ 
via the LLL Algorithm}

\author{Michel L. Lapidus}
\address{Department of Mathematics, University of California \newline 
\indent Riverside, California 92521, USA}
\email{lapidus@math.ucr.edu}

\author{Machiel van Frankenhuijsen}
\address{Department of Mathematics, Utah Valley University \newline  
\indent Orem, Utah 84058, USA}
\email{vanframa@uvu.edu}

\author{Edward K. Voskanian}
\address{Department of Mathematics \& Statistics, The College of New Jersey \newline 
\indent Ewing, New Jersey 08618, USA}
\email{voskanie@tcnj.edu}

\keywords{Lattice and nonlattice self-similar strings, Diophantine approximation, geometric zeta function, complex dimensions, Dirichlet polynomial, roots of Dirichlet polynomials, lattice case, nonlattice case, Lattice String Approximation (LSA) algorithm, quasiperiodic structure and patterns, simultaneous Diophantine approximation, LLL algorithm, algorithm of Lenstra, Lenstra and Lov\'asz.}

\begin{abstract}
The Lattice String Approximation algorithm (or LSA algorithm) of M. L. Lapidus and M. van Frankenhuijsen is a procedure that approximates the complex dimensions of a nonlattice self-similar fractal string by the complex dimensions of a lattice self-similar fractal string. The implication of this procedure is that the set of complex dimensions of a nonlattice string has a quasiperiodic pattern. Using the LSA algorithm, together with the multiprecision polynomial solver MPSolve which is due to D. A. Bini, G. Fiorentino and L. Robol, we give a new and significantly more powerful presentation of the quasiperiodic patterns of the sets of complex dimensions of nonlattice self-similar fractal strings. The implementation of this algorithm requires a practical method for generating simultaneous Diophantine approximations, which in some cases we can accomplish by the continued fraction process. Otherwise, as was suggested by Lapidus and van Frankenhuijsen, we use the LLL algorithm of A. K. Lenstra, H. W. Lenstra, and L. Lov\'asz. 
\end{abstract}

\maketitle

\section{Introduction} \label{introduction}
From 1991 to 1993, Lapidus (in the more general and higher-dimensional case of fractal drums), as well as Lapidus and Pomerance established connections between complex dimensions and the theory of the Riemann zeta function by studying the connection between fractal strings and their spectra; see \cite{lapidus1991fractal}, \cite{lapidus1993vibrations} and \cite{lapidus1993the}. Then, in \cite{lapidus1995the}, Lapidus and Maier used the intuition coming from the notion of complex dimensions in order to rigorously reformulate the Riemann hypothesis as an inverse spectral problem for fractal strings. The notion of complex dimensions was precisely defined and the corresponding rigorous theory of complex dimensions was fully developed by Lapidus and van Frankenhuijsen, for example in \cite{lapidus2000fractal,lapidus2003complex,lapidus2006fractal,lapidus2012fractal}, in the one-dimensional case of fractal strings. Recently, the higher-dimensional theory of complex dimensions was fully developed by Lapidus, Radunovi\'c and {\v Z}ubrini\'c in the book \cite{lapidus2017fractal} and a series of accompanying papers; see also the first author's recent survey article \cite{lapidus2019an}.

The present paper focuses, in particular, on self-similar strings (and their natural generalizations), the boundary of which is a self-similar set 
(in $\mathbb{R}$), 
satisfying a mild non-overlapping condition, as introduced and studied in \cite{lapidus2000fractal,lapidus2003complex,lapidus2006fractal,lapidus2012fractal}.

Given a closed, bounded and nonempty interval 
$I$ 
of length 
$L$, 
and 
$M \geq 2$ 
contraction similitudes of 
$\mathbb{R}$, 
\[ \Phi_1, \dots, \Phi_M:I \to I, \] 
a {\it self-similar fractal string} (or self-similar string, in short) is constructed through a procedure reminiscent of the construction of the Cantor set. In the first step, one subdivides the interval $I$ into the images 
\begin{equation} \label{pieces} 
\Phi_1(I), \dots, \Phi_M(I).
\end{equation}
If one imposes a mild separation condition on the contraction similitudes, the images (\ref{pieces}) now lie in the interval $I$, and they do not overlap, except possibly at the endpoints. Moreover, the complement in $I$ of their union consists of one or more disjoint open intervals, called the {\it first lengths}. This process is then repeated with each of the images in (\ref{pieces}), resulting in another finite collection of disjoint open intervals. The final result will be a countably infinite collection of pairwise disjoint and bounded open intervals, all contained in the original interval $I$. The union of these open intervals is the self-similar fractal string. 

From the perspective of the current paper, there is an important dichotomy in the set of all self-similar fractal strings, according to which any self-similar fractal string is either {\it lattice} or {\it nonlattice}, depending on the scaling ratios with which a self-similar fractal string is constructed.\footnote{Note the two meanings of `lattice'. On the one hand, a lattice string is a certain kind of fractal string studied in fractal geometry. On the other hand, the LLL algorithm is a generalization of Euclid's algorithm aimed at finding a reduced basis of a lattice as a discrete subgroup of $\mathbb{R}^N$ of rank $N$; see Section \ref{section 3.1}.} More specifically, the {\it lattice} (resp., {\it nonlattice}) {\it case} is when all (resp., two or more) of the logarithms of the 
$N$ 
{\it distinct} scaling ratios are rationally dependent (resp., independent), with necessarily 
$1 \leq N \leq M$. 
In other words, the multiplicative group 
$G \subseteq (0,+\infty)$ 
generated by the 
$N$ 
distinct scaling ratios is of rank 1 in the lattice case 
(that is, 
$G = r^\mathbb{Z}$, 
for some 
$r \in (0,1)$, 
called the multiplicative generator) 
and is of rank 
$\geq 2$, 
in the nonlattice case. By definition, the {\it generic nonlattice case} is when 
$N \geq 2$ 
and the rank of 
$G$ 
is equal to~ 
$N$.

In the lattice case, the complex dimensions\footnote{See Section 2 for a reminder of the definition of the complex dimensions as the poles of the `geometric zeta function' associated with a fractal string.} can be numerically obtained via the roots of certain polynomials that are typically sparse with large degrees, and lie periodically on finitely many vertical lines counted according to multiplicity. Furthermore, on each vertical line, they are separated by a positive real number ${\bf p}$, called the oscillatory period of the string. (See \cite[Chapter 2]{lapidus2000fractal}, \cite[Theorem 2.5]{lapidus2003complex}, and \cite[Theorems 2.16 and 3.6]{lapidus2012fractal}.) 

For nonlattice self-similar fractal strings, which are the main focus of the present paper, the complex dimensions cannot be numerically obtained in the same way as in the lattice case. Indeed, they correspond to the roots of a transcendental (rather than polynomial) equation. They can, however, be approximated by the complex dimensions of a sequence of lattice strings with larger and larger oscillatory periods.  The {\it Lattice String Approximation algorithm} of Lapidus and van Frankenhuijsen, referred to in this paper as the LSA {\it algorithm}, allows one to replace the study of nonlattice self-similar fractal strings by the study of suitable approximating sequences of lattice self-similar fractal strings. Using this algorithm, M. L. Lapidus and M. van Frankenhuijsen have shown that the sets of complex dimensions of nonlattice self-similar fractal strings are quasiperiodically distributed, in a precise sense (see, e.g., \cite[Theorem 3.6, Remark 3.7]{lapidus2003complex} and \cite[Section 3.4.2]{lapidus2006fractal,lapidus2012fractal}), and they have illustrated their results by means of a number of examples (see, e.g., the examples from Section 7 in \cite{lapidus2003complex} and their counterparts in Chapters 2 and 3 of \cite{lapidus2012fractal}). Following the suggestion by those same authors in the introduction of \cite{lapidus2003complex}, and in \cite[Remark 3.38]{lapidus2006fractal,lapidus2012fractal}, the current paper presents an implementation of the LSA algorithm incorporating the application of a powerful lattice basis reduction algorithm, which is due to A. K. Lenstra, H. W. Lenstra and L. Lov\'asz and is known as the LLL {\it algorithm}, in order to generate simultaneous Diophantine approximations; see \cite[Proposition 1.39]{lenstra1982factoring} and \cite[Proposition 9.4]{bremner2011lattice}. It also uses the open source software MPSolve, due to D. A. Bini, G. Fiorentino and L. Robol in \cite{bini2000design}, \cite{bini2014solving}, in order to approximate the roots of large degree sparse polynomials. Indeed, the LLL algorithm along with MPSolve allow for a deeper numerical and visual exploration of the quasiperiodic patterns of the complex dimensions of self-similar strings via the LSA algorithm than what has already been done in \cite{lapidus2000fractal, lapidus2003complex, lapidus2006fractal, lapidus2012fractal}.

In the latter part of \cite[Chapter 3]{lapidus2012fractal}, a number of mathematical results were obtained concerning either the nonlattice case with two distinct scaling ratios (amenable to the use of continued fractions) as well as the nonlattice case with three or more distinct scaling ratios (therefore, typically requiring more complicated simultaneous Diophantine approximation algorithms). In the present paper, it has become possible, in particular, to explore more deeply and accurately additional nonlattice strings with rank greater than or equal to three, i.e., those that cannot be solved using continued fractions

The rest of this paper is organized as follows. Section \ref{section 2} consists of some background on complex dimensions and self-similar fractal strings, leading up to the restatement of \cite[Theorem 3.18]{lapidus2012fractal} (see also \cite[Theorem 3.6]{lapidus2003complex}), which provides the LSA algorithm. Then, in Section \ref{section 3}, a brief overview of lattice basis reduction is given, along with a restatement and proof of \cite[Proposition 9.4]{bremner2011lattice} to illustrate how the LLL algorithm is applied to simultaneous Diophantine approximations. In the latter part of Section \ref{section 3} (see Section \ref{section 3.2.1}), we describe our implementation of the LLL algorithm for simultaneous Diophantine approximations which uses continued fractions. In Section \ref{section 4}, using our implementation of the LLL algorithm, together with MPSolve, a number of examples aimed at illustrating the quasiperiodic patterns of the complex dimensions of nonlattice self-similar fractal strings, and in the more general setting, of the roots (i.e., the zeros) of nonlattice Dirichlet polynomials, are shown and commented upon. These include examples previously studied in \cite{lapidus2000fractal,lapidus2003complex,lapidus2006fractal,lapidus2012fractal}, which can now be viewed in a new light by using our refined numerical approach, and new handpicked examples which are computationally easier to explore and for which interesting new phenomena arise. The mathematical experiments performed in the current paper, along with earlier work in \cite{lapidus2008in} and \cite{lapidus2003complex, lapidus2006fractal,lapidus2012fractal}, have led to new questions and open problems which are briefly discussed in the concluding comments section, namely, Section \ref{section 5}.

\section{Preliminary materials} \label{section 2}
An \textit{\textup(ordinary\textup) fractal string} 
$\mathcal{L}$ 
consists of a bounded open subset 
$\Omega \subset \mathbb{R}$; 
such a set 
$\Omega$
 is a disjoint union of countably many disjoint open intervals. The lengths 
\[ \ell_1, \ell_2, \ell_3, \dots \] 
of the open intervals are called the \textit{lengths} of 
$\mathcal{L}$, 
and since 
$\Omega$ 
is a bounded set, it is assumed without loss of generality that
\[ \ell_1 \geq \ell_2 \geq \cdots > 0, \]
and that 
$\ell_j \to 0$ 
as 
$j \to \infty$.
\footnote{We ignore here the trivial case when $\Omega$ is a finite union of open intervals.}  

Let 
$\sigma_{\mathcal{L}}$ 
denote the \textit{abscissa of convergence},\footnote{Note that $\lvert \ell_j^s \rvert = \ell_j^{\Re(s)}$, for every $s \in \mathbb{C}$ and all $j \in \mathbb{N}$.}  
\begin{align*}
\sigma_{\mathcal{L}} &= \inf\left\{\sigma \in \mathbb{R} : \sum_{j = 1}^\infty \lvert \ell_j^s \rvert < \infty, \text{ for every }s \in \mathbb{C}\text{ with }\Re(s) > \sigma\right\} \\
&= \inf\left\{\alpha \in \mathbb{R} : \sum_{j = 1}^\infty \ell_j^\alpha < \infty\right\},
\end{align*}
of the \textit{geometric zeta function} 
\[ \zeta_{\mathcal{L}}(s) = \sum_{j = 1}^\infty \ell_j^s \]
of 
$\mathcal{L}$. 
Since there are infinitely many lengths, 
$\zeta_{\mathcal{L}}(s)$ 
diverges at 
$s = 0$. 
Also, since 
$\Omega$ 
has finite Lebesgue measure, 
$\zeta_{\mathcal{L}}(s)$ 
converges at 
$s = 1$. 
Hence, it follows from standard results about general Dirichlet series (see, e.g.,\cite{serre1973a}) that the second equality in the above definition of $\sigma_{\mathcal{L}}$ holds, and, therefore, that 
$0 \leq \sigma_{\mathcal{L}} \leq 1$.

\begin{definition} \label{dimension of fractal string}
The {\it dimension} of a fractal string 
$\mathcal{L}$ 
with associated bounded open set 
$\Omega$, 
denoted by 
$D_\mathcal{L}$, 
is defined as the \textit{\textup(inner\textup) Minkowski dimension} of 
$\Omega$:
\[ D_{\mathcal{L}} = \inf\{\alpha \geq 0 \colon V(\varepsilon) = O(\varepsilon^{1 - \alpha})\text{, as }\varepsilon \to 0^+\}, \]
where 
$V(\varepsilon)$ 
denotes the volume (i.e., total length) of the \textit{inner tubular neighborhood} of 
$\partial\Omega$ 
with radius 
$\varepsilon$ given by
\[ V(\varepsilon) = \operatorname{vol}_1\left(\{x \in \Omega \colon d(x,\partial\Omega) < \varepsilon\}\right). \]   
\end{definition}

According to \cite[Theorem 1.10]{lapidus2012fractal} (see also \cite{lapidus1993vibrations}), the abscissa of convergence 
$\sigma_{\mathcal{L}}$ 
of a fractal string 
$\mathcal{L}$ 
coincides with the dimension 
$D_{\mathcal{L}}$ of $\mathcal{L}$: $\sigma_{\mathcal{L}} = D_{\mathcal{L}}$.

\begin{definition} \label{complex dimensions}
Suppose 
$\zeta_{\mathcal{L}}(s)$ 
has a meromorphic continuation to the entire complex plane. Then the poles of 
$\zeta_{\mathcal{L}}(s)$ 
are called the \textit{complex dimensions} of 
$\mathcal{L}$. 
\end{definition}

\begin{remark}
While the theory of complex dimensions is developed in \cite{lapidus2000fractal,lapidus2003complex,lapidus2006fractal,lapidus2012fractal} and \cite{lapidus2017fractal} for geometric zeta functions not necessarily having a meromorphic continuation to all of 
$\mathbb{C}$, 
the present paper only requires the simpler case considered in Definition \ref{complex dimensions} (which is the case, in particular, of all self-similar fractal strings). See also \cite{lapidus2019an} for a recent survey of the theory of complex fractal dimensions.
\end{remark}

The geometric importance of the set of complex dimensions of a fractal string 
$\mathcal{L}$ 
with boundary 
$\Omega$, 
which always includes its inner Minkowski dimension 
$D_{\mathcal{L}}$, 
is justified because, for example, the complex dimensions appear in an essential way in the explicit formula for the volume 
$V(\varepsilon)$ 
of the inner tubular neighborhood of the boundary 
$\partial\Omega$; 
see the corresponding ``fractal tube formulas'' obtained in Chapter 8 of \cite{lapidus2012fractal}. Accordingly, the complex dimensions give very detailed information about the intrinsic oscillations that are inherent to fractal geometries; see also Remark \ref{definition of fractal} below. The current paper, however, deals with the complex dimensions viewed only as a discrete subset of the complex plane, and the focus is on the special type of fractal strings that are constructed through an iterative process involving scaling, as is discussed in Section \ref{section 2.1}.

\begin{remark} \label{definition of fractal}
In \cite{lapidus2000fractal,lapidus2003complex,lapidus2006fractal,lapidus2012fractal} 
(when $\nu = 1$) 
and in \cite{lapidus2017fractal,lapidus2019an} 
(when the integer $\nu \geq 1$ is arbitrary), 
a geometric object is said to be {\it fractal} if it has at least one nonreal complex dimension.\footnote{It then has at least two nonreal complex dimensions since, clearly, nonreal complex dimensions come in complex conjugate pairs.} This definition applies to fractal strings (including all self-similar strings, which are shown to be fractal in this sense)\footnote{In fact, self-similar strings have infinitely many nonreal complex dimensions; see, e.g., Equation (2.37) in \cite[Theorem 2.16]{lapidus2012fractal}.}, that correspond to the 
$\nu = 1$ 
case, and to bounded subsets of 
$\mathbb{R}^\nu$ (for any integer $\nu \geq 1$) 
as well as, more generally, to relative fractal drums, which are natural higher-dimensional counterparts of fractal strings.  
\end{remark}

\subsection{Self-Similar Fractal Strings} \label{section 2.1}
\hfill \\
Let 
$I$ 
be a connected interval with length 
$L$. 
Let
\begin{equation} \label{contractions}
\Phi_1, \Phi_2, \dots, \Phi_M:I \to I
\end{equation}
be 
$M \geq 2$
contraction similitudes with distinct scaling ratios
\[ 1 > r_1 > r_2 > \cdots > r_N > 0. \]
What this means is that for all 
$j = 1, \dots, M$,
\[ \left\rvert \Phi_j(x) - \Phi_j(y) \right\lvert = r_j\lvert x - y \rvert, \text{ for all } x, y \in I. \]
Assume that after having applied to $I$ each of the maps 
$\Phi_j$
in (\ref{contractions}),  
for 
$j = 1, \dots, M$, 
the resulting images
\begin{equation} \label{images}
\Phi_1(I), \dots, \Phi_M(I)
\end{equation}
do not overlap, except possibly at the endpoints, and that 
$\sum_{j = 1}^M r_j < 1$.

These assumptions imply that the complement of the union, 
$\bigcup_{j = 1}^M \Phi_j(I)$, 
in 
$I$ 
consists of 
$K$ 
pairwise disjoint open intervals with lengths 
\[ 1 > g_1L \geq g_2L \geq \cdots \geq g_KL > 0, \]
called the {\it first intervals}. Note that the quantities 
$g_1, \dots, g_K$, 
called the \textit{gaps}, along with the scaling ratios $r_1, \dots, r_M$, satisfy the equation 
\[ \sum_{j = 1}^M r_j + \sum_{k = 1}^K g_k = 1. \]

The process which was just described above is then repeated for each of the 
$M$ 
images 
$\Phi_j(I)$, 
for 
$j = 1, \dots, M$, 
in (\ref{images}) in order to produce 
$KM$ 
additional pairwise disjoint open intervals in 
$I$. 
Repeating this process ad infinitum yields countably many open intervals, which defines a fractal string 
$\mathcal{L}$ 
with bounded open set 
$\Omega$ 
given by the (necessarily disjoint) union of these open intervals. Any fractal string obtained in this manner is called a \textit{self-similar fractal string} (or a {\it self-similar string}, in short). 

Lapidus and van Frankenhuijsen have shown that the geometric zeta functions of self-similar fractal strings have meromorphic continuations to all of 
$\mathbb{C}$; 
see Theorem 2.3 in Chapter 2 of \cite{lapidus2012fractal}. Specifically, the geometric zeta function 
$\zeta_\mathcal{L}(s)$ 
of any self-similar fractal string with scaling ratios 
$\{r_j\}_{j = 1}^M$, 
gaps 
$\{g_k\}_{k = 1}^K$, 
and total length 
$L$ 
is given by
\begin{equation} \label{extension} 
\zeta_\mathcal{L}(s) = \frac{L^s\sum_{k = 1}^{K} g_k^s}{1 - \sum_{j = 1}^{M} r_j^s}, \text{ for all } s \in \mathbb{C}.
\end{equation}
Both the numerator and the denominator of the right-hand side of (\ref{extension}) are special kinds of exponential polynomials, known as Dirichlet polynomials. Hence, as was done in \cite{lapidus2003complex} and \cite[Chapter 3]{lapidus2012fractal}, the more general situation of the sets of complex roots of Dirichlet polynomials is considered, as is next explained. 

\begin{definition} \label{Dirichlet polynomial}
Given an integer $N \geq 1$, let $r_0 > r_1 > \cdots > r_N > 0$, and let $m_0, m_1, \dots, m_N \in \mathbb{C}$. The function $f: \mathbb{C} \to \mathbb{C}$ given by 
\begin{equation} \label{poly} 
f(s) = \sum_{j = 0}^N m_jr_j^s
\end{equation}
is called a \textit{Dirichlet polynomial} with {\it scaling ratios} 
$r_1, \dots, r_N$
and respective {\it multiplicities} 
$m_0, \dots, m_N$.
\footnote{In the geometric situation of a self-similar string $\mathcal{L}$ discussed just above, $r_0 := -1$ and $m_0 := 1$, while the $r_j$'s, with $j = 1, \dots, N$, correspond to the {\it distinct} scaling ratios, among the scaling ratios $\{r_j\}_{j = 1}^M$ of $\mathcal{L}$. Hence, in particular, $1 \leq N \leq M$ in this case, and, modulo a suitable abuse of notation, for each distinct scaling ratio $r_j$, for $j = 1, \dots, M$, $m_j := \#\{1 \leq k \leq M\colon r_k = r_j\}$ is indeed the multiplicity of $r_j$.} 
\end{definition}

Therefore, the set of complex dimensions of any self-similar fractal string is a subset of the set of complex roots of an associated Dirichlet polynomial 
$f(s)$, 
as given in (\ref{poly}). While, in general, some of the zeros of the denominator of the right-hand side of (\ref{extension}) could be cancelled by the roots of its numerator (see \cite[Section 2.3.3]{lapidus2012fractal}), in the important special case of a single gap length 
(i.e., when $g_1 = \cdots = g_K$), 
the complex dimensions precisely coincide with the complex roots of 
$f(s)$. 
This can be seen directly (in light of (\ref{extension})) or else by choosing the length 
$L$ 
of the interval to be the reciprocal of the single gap length, which simplifies the geometric zeta function in such a way that the numerator, on the right-hand side of (\ref{extension}), is equal to 1; note that this rescaling has no effect on the complex dimensions. Hence, in that case, there are no cancellations, and all of the roots of 
$f(s)$ 
are complex dimensions.

The result in (\ref{extension}) establishes a deep connection between the study of complex dimensions of self-similar fractal strings and that of the roots of Dirichlet polynomials which was gaining interest as early as the start of the nineteenth century; see, e.g.,  \cite{lapidus2003complex,lapidus2006fractal,lapidus2012fractal} and \cite{mora2013on,dubon2014on}, along with the relevant references therein.

In light of the discussion surrounding Equation (\ref{extension}), it suffices to study more generally the sets of complex roots of Dirichlet polynomials, which will be the focus for the remainder of this paper. For the purpose of investigating the sets of complex roots of Dirichlet polynomials, it is assumed without any loss of generality that 
$m_0 := -1$ 
and 
$r_0 := 1$ 
in (\ref{poly}). That is, in the remainder of this paper, we will only consider Dirichlet polynomials of the form

\begin{equation} \label{normalized Dirichlet polynomial}
f(s) = 1 - \sum_{j = 1}^N m_jr_j^s,
\end{equation}
with 
$\{r_j\}_{j = 1}^N$ 
and 
$\{m_j\}_{j = 1}^N$ 
as in Definition \ref{Dirichlet polynomial}.

\subsection{Lattice/Nonlattice Dichotomy and Lattice String Approximation} \label{section 2.2}
\hfill \\
Let 
$f$ 
be a Dirichlet polynomial, with distinct scaling ratios 
$r_1, \dots, r_N$ 
and multiplicities 
$m_1, \dots, m_N$, 
given by (\ref{normalized Dirichlet polynomial}). Define the {\it weights} 
$w_1, \dots, w_N$ 
of 
$f$ 
by 
$w_j: = -\log r_j$, 
for 
$1 \leq j \leq N$. 

\begin{definition} \label{lattice/nonlattice}
A Dirichlet polynomial $f$ is called {\it lattice} if $w_j / w_1$ is rational for $1 \leq j \leq N$, and it is called {\it nonlattice} otherwise.\footnote{Note that if $N = 1$, then $f$ must be lattice because $w_1/w_1 = 1$ is rational.}
\end{definition}

It is straightforward to check that a Dirichlet polynomial 
$f$ 
is lattice if and only if there exists a (necessarily unique) real number 
$r$ in 
$(0,1)$, 
called the \textit{multiplicative generator of 
$f$}, 
and positive integers 
$k_1,\dots,k_N$, 
without common divisors, such that 
$r_j = r^{k_j}$ for $j = 1, \dots, N$. 
Put another way, the lattice case is when the rank of the additive group
\[ G := \sum_{j = 1}^N \mathbb{Z}w_j \]
equals 1, and the nonlattice case is when this rank is 
$\geq 2$.

\begin{definition} \label{rank and generic}
The {\it rank} of a Dirichlet polynomial is defined to be the rank of the group $G$ defined above. Then, $f(s)$ is called {\it generic nonlattice} if the number $N$ of distinct scaling ratios satisfies $N \geq 2$ and is equal to the rank of $f(s)$; furthermore, still of $N \geq 2$, $f(s)$ it is said to be {\it nongeneric nonlattice}, otherwise. In other words, $f(s)$ is generic nonlattice if and only if $N \geq 2$ and $w_1, \dots, w_N$ are rationally independent. 
\end{definition}

Moreover, a self-similar fractal string is called (generic) nonlattice if its associated Dirichlet polynomial is (generic) nonlattice. A thorough description of the structure of the sets of complex roots of Dirichlet polynomials is provided by \cite[Theorem 3.6]{lapidus2012fractal}. Some of the most relevant features to the current paper are as follows:

The set of complex roots of any Dirichlet polynomial is a subset of the horizontally bounded vertical strip 
\[ R := \{z \in \mathbb{C} : D_\ell \leq \Re z \leq D\}, \] 
where 
$D_\ell$ 
and 
$D$ 
are the unique real numbers satisfying the equations\footnote{If $N = 1$, then the second sum on the left-hand side of (\ref{sums}) is equal to zero, by convention.}
\begin{equation} \label{sums}
1 + \sum_{j = 1}^{N - 1} |m_j|r_j^{D_\ell} = |m_{N}|r_N^{D_\ell}\quad \text{and}\quad \sum_{j = 1}^{N} |m_j|r_j^D = 1,
\end{equation}
respectively. These numbers satisfy the inequality 
$-\infty < D_{\ell} \leq D$. 
If the multiplicities are positive integers, then the complex roots are symmetric about the real axis, the number 
$D$ 
defined above is positive, and it is the only real root of 
$f$; 
furthermore, it is a simple root.

\begin{remark}
In the case of a self-similar string 
$\mathcal{L}$, 
the nonnegative number 
$D$ 
does not exceed 
$1$ 
and coincides with 
$D_{\mathcal{L}}$, 
the inner Minkowski dimension of 
$\mathcal{L}$: $D = D_{\mathcal{L}} = \sigma_{\mathcal{L}}$, 
in the notation introduced earlier for fractal strings.
\end{remark}
 
\subsubsection{Lattice Versus Nonlattice} \label{section 2.2.1}
If the Dirichlet polynomial (\ref{normalized Dirichlet polynomial}), with distinct scaling ratios
\[ 1 = r_0 > r_1 > \cdots > r_N > 0, \]
is lattice, then according to Definition \ref{lattice/nonlattice}, the associated real numbers
\begin{equation} \label{quotients}
1 < \frac{w_2}{w_1} < \cdots < \frac{w_N}{w_1} < \infty,
\end{equation}
which are explicitly determined by the weights
\[ 0 = w_0 < w_1 < \cdots < w_N < \infty, \] 
are all rational. Therefore, there exist positive integers
\[ q < k_2 < \cdots < k_N < \infty \]
such that 
\[ \frac{w_j}{w_1} = \frac{k_j}{q}, \quad \text{for } j = 2, \dots, N. \]

According to \cite[Theorem 3.6]{lapidus2012fractal}, the complex roots of a lattice Dirichlet polynomial 
$f(s)$ 
lie periodically on finitely many vertical lines, and on each line they are separated by the positive number 
\[ \textbf{p} = \frac{2\pi}{\log r^{-1}}, \] 
called the 
\textit{oscillatory period of $f(s)$}. 

More precisely, following the discussion surrounding \cite[Equation (2.48), p. 58]{lapidus2012fractal}, the roots are computed by first rewriting 
$f(s)$ 
as a polynomial $g(z)$ of degree 
$k_N$ 
in the complex variable 
$z := r^s$, 
where 
$r = r_1^{1/q}$ 
is the multiplicative generator of 
$f(s)$: 
\begin{equation} \label{complex poly}
g(z) = 1 - m_1z^q - m_2z^{k_2} - m_3z^{k_3} - \cdots - m_Nz^{k_{N}}.
\end{equation}
There are 
$k_N$ 
roots of 
$g(z)$, 
counted with multiplicity. Each one is of the form 
\[ z = \lvert z \rvert e^{i\theta}, \]
where 
$-\pi < \theta \leq \pi$, 
and it corresponds to a unique root of $f(s)$, namely,
\[ \omega = \frac{-\log \lvert z \rvert}{\log r^{-1}} - \frac{i \theta}{\log r^{-1}}. \]

Therefore, given a lattice Dirichlet polynomial 
$f$ 
with oscillatory period 
$\textbf{p}$, 
there exist complex numbers 
$\omega_1, \dots, \omega_u$ 
such that the set 
$\mathcal{D}_f$ 
of complex roots of $f$ is given by
\[ \mathcal{D}_f = \bigcup_{1 \leq j \leq u} H_j, \]
where for 
$1 \leq j \leq u$,
\[ H_j := \{\omega_j + in\textbf{p} : n \in \mathbb{Z}\}. \]
Now, let 
$f$ 
be a nonlattice Dirichlet polynomial given by
\[ f(s) = 1 - \sum_{j = 1}^N m_jr_j^s \]
and with weights 
$w_1, \dots, w_N$. 
Then, 
$N \geq 2$ 
and at least one of the associated real numbers
\begin{equation}
\frac{w_2}{w_1}, \frac{w_3}{w_1}, \dots, \frac{w_N}{w_1}
\end{equation}
is irrational; see Definition \ref{lattice/nonlattice}. Writing
\begin{equation} \label{nonlattice}
f(s) = 1 - \sum_{j = 1}^{N} m_j(r_1^s)^{w_j/w_1},
\end{equation}
and noting that not all of the associated real numbers in (\ref{quotients}) are rational, one sees that 
$f(s)$ 
cannot be expressed as a polynomial, which means that the approach which was just described above to compute the roots of a lattice Dirichlet polynomial is not applicable in the nonlattice case. Instead, the roots are approximated by a procedure developed by the two authors of \cite{lapidus2000fractal,lapidus2003complex,lapidus2006fractal,lapidus2012fractal}. The practicality of their procedure rests upon a long standing-problem in the theory of Diophantine approximations: namely, to efficiently generate infinitely many good rational approximations, with a common denominator, to a vector of real numbers with at least one irrational coordinate; see, e.g., \cite{lagarias1982best}.

Using this approximation procedure, referred to in the current paper as the {\it Lattice String Approximation} algorithm (LSA algorithm for short), the authors of \cite{lapidus2000fractal,lapidus2003complex,lapidus2006fractal,lapidus2012fractal} showed that the set of complex roots of nonlattice Dirichlet polynomial equations exhibit continuously evolving periodic patterns; they say that the complex roots of nonlattice Dirichlet polynomials have a {\it quasiperiodic pattern}; see, e.g., Remark 3.7 in \cite{lapidus2003complex} and Section 3.4 in \cite{lapidus2012fractal}.
  
\subsubsection{Lattice String Approximation} \label{section 2.2.2}
The LSA algorithm is based on the theory of Diophantine approximations, which deals with the approximation of real numbers by rational numbers. The main step of the algorithm replaces the real numbers $w_2/w_1, \dots, w_N/w_1$ in (\ref{quotients}) with rational approximations having a common denominator. The resulting Dirichlet polynomial is lattice, and the roots of this resulting Dirichlet polynomial approximate the roots of (\ref{nonlattice}) in a very special way. Therefore, before stating Theorem \ref{approximation theorem} below, which provides the algorithm, the following lemma on simultaneous Diophantine approximation is stated; see, e.g., \cite[Theorem 1A, p. 27]{schmidt1980diophantine}. This well-known result is a generalization to several real numbers of Dirichlet's approximation theorem, itself a consequence of the pigeonhole principle.

\begin{lemma} \label{simultaneous approximation}
Let $f(s)$ be a nonlattice Dirichlet polynomial with associated weights $w_1, \dots, w_N$, where $N \geq 2$. Then, for every real number $Q > 1$, there exist infinitely many vectors $(q,k_2,\dots,k_N) \in \mathbb{N}^{N + 1}$ such that
\begin{equation} \label{corollary to inequality}
\left\lvert \frac{w_{j}}{w_1} - \frac{k_j}{q} \right \rvert < \frac{1}{qQ},
\end{equation}
for all integers $j$ such that $2 \leq j \leq N$ and integers $q$ such that $1 \leq q < Q^{N - 1}$. Moreover, $q \to \infty$ as $Q \to \infty$. 
\end{lemma}

Let 
$f$ 
be a nonlattice Dirichlet polynomial with associated real numbers as in 
$(\ref{quotients})$. 
If a real number 
$Q > 1$ 
and positive integers 
$q, k_2, \dots, k_N$ 
are computed such that 
$Q$ and the vector 
$(q, k_2, \dots, k_N)$ 
satisfy inequality (\ref{corollary to inequality}), for each integer $j \in \{ 2, \dots, N\}$, then the pair
\begin{equation} \label{simultaneous Diophantine approximation}
Q,\ (q,k_2,k_3, \dots, k_N)
\end{equation}
is called a {\it simultaneous Diophantine approximation} to the associated real numbers in (\ref{quotients}).\footnote{In the standard literature, the denominator $q$ in a simultaneous Diophantine approximation also satisfies the inequality $q < Q^{N - 1}$ in Lemma \ref{simultaneous approximation}. However, since the proof of Theorem \ref{approximation theorem} does not use this estimate, we ignore this fact here and in the sequel.} We are now ready to state the following key result (\cite[Theorem 3.18, p. 34]{lapidus2012fractal}).

\begin{theorem}[\textup{M. L. Lapidus and M. van Frankenhuijsen; \cite{lapidus2000fractal,lapidus2003complex,lapidus2006fractal,lapidus2012fractal}}]
\label{approximation theorem}
Let $f(s)$ be a nonlattice Dirichlet polynomial of the form \textup{(\ref{normalized Dirichlet polynomial})} with scaling ratios $1 > r_1 > \cdots > r_N > 0$ and multiplicities $m_1, \dots, m_N$, where $N \geq 2$.\footnote{Recall that one must then have $N \geq 2$; otherwise, one would be in the lattice case.} 
Let $Q > 1$, and let $q$ and $k_j$ be as in \textup{Lemma \ref{simultaneous approximation}} \textup(except possibly without the condition $q < Q^{N - 1}$\textup). 
Then, the Dirichlet polynomial 
\[ f_q(s) = 1 - m_1\left(r_1\right)^s - \sum_{j = 2}^{N} m_j\left(r_1^{k_j/q}\right)^s \] 
is lattice with generator $r_1^{1/q}$ and oscillatory period 
\[ {\bf p} = {\bf p}_q =  := \frac{2\pi q}{\log r_1^{-1}}. \] 
Moreover, for every {\em approximation error} $\varepsilon > 0$, if $s$ belongs to the $\varepsilon$-{\em region of stability} \textup(\em{of radius} $\varepsilon$CQ{\bf p}\textup)
\[ B_\varepsilon(q,Q) := \left\{z \in \mathbb{C}\colon |z| < \varepsilon CQ{\bf p}\right\}, \]
then 
\[ |f_q(s) - f(s)| < \varepsilon, \]
where
\begin{equation} \label{LSA constant}
C := \frac{1}{2\pi}\sum_{j = 1}^N \lvert m_j \rvert\left(\frac{\sum_{j = 0}^N \lvert m_j \rvert}{\min\left\{1,\lvert m_N \rvert\right\}}\right)^{\frac{-2w_N}{\min\left\{w_1,w_{N} - w_{N - 1}\right\}}}
\end{equation}
is the \emph{LSA constant of $f(s)$}.\footnote{Note that since $m_0 = -1$, we have that $\lvert m_0 \rvert = 1$ here.}
\end{theorem}

\begin{remark} \label{LSA terminology}
For a fixed order of accuracy 
$\varepsilon > 0$, 
we call a root of a lattice string approximation 
$f_q(s)$ 
lying inside its 
$\varepsilon$-region of stability 
a {\it stable} root, and we say that 
$f_{q_2}(s)$ 
is {\it more stable} than 
$f_{q_1}(s)$ 
if the $\varepsilon$-region of stability 
of 
$f_{q_2}(s)$ 
contains that of 
$f_{q_1}(s)$.
\end{remark}

In summary, Theorem \ref{approximation theorem} says that a simultaneous Diophantine approximation 
\[ Q,\quad (q, k_2, \dots, k_N) \]
to the real numbers in (\ref{quotients}) determines a {\it lattice string approximation} 
\begin{equation} \label{lsa approximation}
f_q(s) = 1 - r_1^{-s} - r_1^{-sk_2/q} - \cdots - r_1^{-sk_N/q},
\end{equation}
with the property that its values are close to the values of $f(s)$, with prescribed approximation error 
$\varepsilon > 0$, 
within a region of stability with radius that is proportional to 
$\varepsilon$; 
the smaller the approximation error, the smaller the region of stability. The implication of Theorem \ref{approximation theorem} is that the roots of 
$f(s)$ 
are {\it almost periodically distributed}. That is, given a lattice string approximation $f_q(s)$ to a nonlattice Dirichlet polynomial $f(s)$, the roots of $f(s)$ are near the periodically distributed roots of $f_q(s)$, for a certain number of periods. Then, the roots of $f(s)$ start to deviate from this periodic pattern, and a new periodic pattern, associated with a more stable lattice string approximation, gradually emerges; see {\it ibid}. 

Before illustrating this discussion by means of several examples in Section \ref{section 4}, we present our implementation of the LLL algorithm for simultaneous Diophantine approximations in the following section, which we will use in order to explore the roots of the more complicated Dirichlet polynomials with rank three or more. 

\section{Simultaneous Diophantine Approximations} \label{section 3}
In general, approximating the sets of complex roots of a nonlattice Dirichlet polynomial via the LSA algorithm requires a practical method for generating simultaneous Diophantine approximations. In 1982, A. K. Lenstra, H. W. Lenstra, and L. Lov\'asz (or LLL, for brevity) presented in their paper \cite{lenstra1982factoring} the first polynomial-time algorithm to factor a nonzero polynomial 
$f \in \mathbb{Q}[x]$
into irreducible factors in 
$\mathbb{Q}[x]$. 
Specifically, the number of arithmetic operations needed is bounded by a constant multiple of 
$n^4$, 
where 
$n$ 
is the size of the input; see \cite[Proposition 1.26]{lenstra1982factoring}. The authors of that paper showed that their algorithm, which is now commonly referred to as the LLL algorithm, can generate simultaneous Diophantine approximations; see Theorem \ref{LLL theorem} below. As suggested by Lapidus and van Frankenhuijsen in \cite[p. 99]{lapidus2006fractal}, and then later in \cite[Remark 3.38, p. 101]{lapidus2012fractal}, the current paper utilizes the LLL algorithm in order to generate lattice string approximations.

The present section gives an overview of the LLL algorithm, and explains how it can be used to generate simultaneous Diophantine approximations. For more detail on the LLL algorithm and the corresponding method of lattice basis reduction, the interested reader can consult the original paper, \cite{lenstra1982factoring}, together with Bremner's book, \cite{bremner2011lattice}, providing an introductory exposition of the algorithm.

\subsection{Lattice Basis Reduction} \label{section 3.1}
\begin{definition} \label{definition of a full rank lattice}
Let 
$n$ 
be a positive integer. A subset 
$L$ 
of the 
$n$-dimensional 
real vector space 
$\mathbb{R}^n$ 
is called a (full-rank) \textit{lattice} if there exists a basis 
$\beta = \{{\bf x}_1, \dots, {\bf x}_n\}$ of $\mathbb{R}^n$ 
such that 
\[ L = \left\{\sum_{j = 1}^n a_j{\bf x}_j : a_1, a_2, \dots, a_n \in \mathbb{Z} \right\}. \]
The subset 
$\beta$ 
is called a \textit{basis of} 
$L$, 
and 
$n$ 
is called the \textit{rank of} 
$L$. 
Moreover, for each 
$1 \leq j \leq n$, 
let 
\[ {\bf x}_j = (x_{j,1}, \dots, x_{j,n}), \] 
where for each 
$1 \leq k \leq n$, 
$x_{j,k} \in \mathbb{R}$. 
Then, the \textit{determinant} 
$d(L)$ 
of 
$L$ 
is defined by 
$d(L) = \lvert \det(X) \rvert$, 
where 
$X$ 
is the 
$n \times n$ 
matrix given by 
$(X)_{jk} = x_{j,k}$; 
the matrix 
$X$ 
is called the 
\textit{basis matrix of $L$ for $\beta$}.
\end{definition}

It readily follows from the following proposition that the positive number $d(L)$ does not depend on the choice of basis, as is stated in Corollary \ref{well defined} below.  

\begin{proposition} \label{basis matrix}
Let 
$\beta_1 = \{{\bf x}_1, \dots, {\bf x}_n\}$ 
and 
$\beta_2 = \{{\bf y}_1, \dots, {\bf y}_n\}$ 
be two bases of a lattice 
$L$ 
of 
$\mathbb{R}^n$. 
Let 
$X$ 
and 
$Y$ 
be the basis matrices of 
$L$ 
corresponding to the bases 
$\beta_1$ 
and 
$\beta_2$, 
respectively. Then, 
\[ Y = BX, \]
for some 
$n \times n$ 
matrix 
$B$ 
with integer entries and determinant 
$\pm 1$; 
hence, either 
$B$ or \textup(if $n$ is odd\textup) 
$-B$ belongs to 
$SL\textup(n,\mathbb{Z}\textup)$.
\end{proposition}

\begin{proof}
Since 
$\beta_1, \beta_2 \subset L$, 
for each 
$j = 1, \dots, n$, 
there exist integers 
\[ a_{j,1}, \dots, a_{j,n}, b_{j,1}, \dots, b_{j,n}, \] 
such that 
\[ {\bf x}_j = a_{j,1}y_1 + \cdots + a_{j,n}y_n \quad \text{and} \quad {\bf y}_j = b_{j,1}x_1 + \cdots + b_{j,n}x_n.  \]
This means that there exist 
$n \times n$ 
matrices 
$A$ 
and 
$B$ 
with integer entries given by 
\[ (A)_{jk} = a_{jk} \quad \text{and} \quad (B)_{jk} = b_{jk}, \] 
respectively, such that 
$X = AY$ 
and 
$Y = BX$. 
By substitution, 
$Y = (BA)Y$. 
Since 
$Y$ 
is invertible, it follows that 
$BA = I$, 
and so 
$\det(A)\det(B) = 1$. 
Since 
$A$ 
and 
$B$ 
have integer entries, it follows that 
$\det(B) = \pm 1$, 
as desired.  
\end{proof}

\begin{corollary} \label{well defined}
Under the hypotheses and with the notation of Proposition \ref{basis matrix}, 
\[ \det(L) = \lvert \det(X) \rvert = \lvert \det(Y) \rvert. \]
\end{corollary}

Therefore, the determinant of 
$L$, 
$\det(L) = \lvert \det(X) \rvert$, 
is independent of the choice of the basis of 
$L$ 
used to evaluate it.

Suppose that we are given a lattice 
$L \subset \mathbb{R}^n$. 
In a {\it shortest vector problem}, one finds the shortest nonzero vector in 
$L$. 
That is, one tries to compute
\[ \lambda = \lambda(L) := \min_{{\bf x} \in L \setminus \{{\bf 0}\}} \lvert {\bf x} \rvert. \]
In the 
{\it $\gamma$-approximation 
version} of such a problem, one finds a nonzero lattice vector of length at most 
$\gamma \cdot \lambda(L)$, 
for a given real number
$\gamma \geq 1$. 
These types of problems have many applications in number theory and cryptography; see, e.g., \cite[Chapters 7 and 9]{bremner2011lattice}. No efficient algorithm is known to find the shortest vector in a lattice, or even just the length of the shortest vector. The LLL algorithm is the first polynomial-time algorithm to compute what is called an \textit{$\alpha$-reduced basis} for a given lattice; see \cite[Proposition 1.26]{lenstra1982factoring}. Simply put, an $\alpha$-reduced basis for a lattice 
$L$ 
is one with short vectors that are nearly orthogonal. 

\subsubsection{The $\alpha$-Reduced Basis for a Lattice} \label{section 3.1.1}
Let 
$\beta = \{{\bf x}_1, \dots, {\bf x}_n\}$ 
be a basis of $\mathbb{R}^n$, 
and let 
${\bf x}_1^\ast = {\bf x}_1$. 
For 
$1 < j \leq n$, 
define
\[ {\bf x}_j^\ast = {\bf x}_j - \sum_{k = 1}^{j - 1} \mu_{j,k}{\bf x}_k^\ast, \]
where, for 
$1 \leq k < j \leq n$, 
\[ \mu_{j,k} = \frac{{\bf x}_j \cdot {\bf x}_k^\ast}{{\bf x}_k^* \cdot {\bf x}_k^\ast} = \frac{{\bf x}_j \cdot {\bf x}_k^\ast}{\lvert {\bf x}_k^\ast \rvert^2}. \]
The vectors 
${\bf x}_1^\ast, \dots, {\bf x}_n^\ast$, 
called the \textit{Gram--Schmidt orthogonalization} of 
$\beta$, form an orthogonal basis of 
$\mathbb{R}^n$, 
and the numbers 
$\mu_{j,k}$ 
are called the \textit{Gram--Schmidt coefficients} of the orthogonalization. 

\begin{definition} \label{reduced basis}
Let 
$\beta = \{{\bf x}_1, \dots, {\bf x}_n\}$ 
be a basis for a lattice 
$L \subset \mathbb{R}^n$, 
and let 
${\bf x}_1^\ast, \dots, {\bf x}_n^\ast$ 
be its Gram--Schmidt orthogonalization with Gram--Schmidt coefficients 
$\mu_{j,k}$, 
for 
$1 \leq k < j \leq n$. 
Furthermore, let 
$\alpha$ 
be such that 
$1/4 < \alpha < 1$. 
The basis 
$\beta$ 
is said to be 
$\alpha$-\textit{reduced} 
if the following two conditions are satisfied:
\begin{enumerate}
\item[(i)] $\lvert\mu_{j,k}\rvert \leq 1/2,\quad \text{for } 1 \leq k < j \leq n$;
\item[(ii)] $\lvert x_j^\ast + \mu_{j,j - 1}x_{j - 1}^\ast \rvert^2 \geq \alpha\lvert x_{j - 1}^\ast \rvert^2,\quad \text{for } 1 < j \leq n$.
\end{enumerate}
\end{definition}

\subsection{Simultaneous Diophantine Approximations via LLL} \label{section 3.2}
\hfill \\
In this section, we discuss the key steps needed in order to generate lattice string approximations by using the LLL algorithm. Our implementation uses continued fractions, as opposed to using rational numbers with denominator equal to a power of $2$, which was the approach used in \cite[Remark 4.1, p. 177]{bosma2013finding}. 

After having recalled a technical result from \cite{lenstra1982factoring} (Proposition 3.5 below), we state and prove the main result of \cite{lenstra1982factoring}, from our present perspective, namely, Theorem \ref{LLL theorem}. This result establishes the LLL algorithm as a useful tool for computing simultaneous Diophantine approximations to two or more real numbers.

\begin{proposition}[\textup{\cite[Proposition 1.6]{lenstra1982factoring}; as described, e.g., in \cite[Proposition 4.6]{bremner2011lattice}}] \label{proposition for LLL applied to simultaneous Diophantine approximation}
Let $\beta = \{{\bf x}_1, \dots, {\bf x}_n\}$ be an $\alpha$-reduced basis for a lattice $L \subset \mathbb{R}^n$, and let ${\bf x}_1^\ast, \dots, {\bf x}_n^\ast$ be its Gram--Schmidt orthogonalization. Then, the following three properties hold{\textup :}
\begin{enumerate}
\item[(i)] $\lvert {\bf x}_k \rvert^2 \leq (4/(4\alpha - 1))^{j - 1} \cdot \lvert {\bf x}_j^\ast \rvert^2,\quad \text{for } 1 \leq k \leq j \leq n$;
\item[(ii)] $d(L) \leq \Pi_{j = 1}^n \lvert {\bf x}_j \rvert \leq (4/(4\alpha - 1))^{n(n - 1)/4} \cdot d(L)$;
\item[(iii)] $\lvert {\bf x}_1 \rvert \leq (4/(4\alpha - 1))^{\frac{n - 1}{4}}d(L)^{\frac{1}{n}}$.
\end{enumerate}
\end{proposition}

\begin{theorem}[\textup{A. K. Lenstra, H. W. Lenstra and L. Lov\'asz, \cite[Proposition 1.39]{lenstra1982factoring}; as described, e.g., in \cite[Proposition 9.4]{bremner2011lattice}} \label{LLL theorem}]
Given rational numbers 
$x_1, x_2, \dots, x_n$ 
and 
$\delta$ 
satisfying 
$0 < \delta < 1$, 
there exists a polynomial-time algorithm \textup(called the LLL algorithm\textup) which finds integers 
$b \in \mathbb{N}$ 
and 
$a_1, \dots, a_n \in \mathbb{Z}$ 
such that
\begin{equation} \label{LLL theorem inequalities}
\left\lvert x_j - \frac{a_j}{b} \right\rvert \leq \frac{\delta}{b},\quad \textup{and}\quad 1 \leq b \leq 2^{\frac{n(n + 1)}{4}}\delta^{-n} \textup{ for } j = 1, \dots, n.
\end{equation}
\end{theorem}
\begin{proof}
Let 
$L$ 
be the lattice of rank 
$n + 1$ 
with basis matrix
\[ X =
\begin{pmatrix}
2^{\frac{-n(n + 1)}{4}}\delta^{n + 1} & x_1 & \cdots & x_n \\
0 & -1 & \cdots & 0 \\
\vdots & \vdots & \ddots & \vdots \\
0 & 0 & \cdots & -1
\end{pmatrix}
. \]
Using the LLL Algorithm, generate a reduced basis 
$\beta = \{{\bf y}_1, \dots, {\bf y}_{n}\}$ 
for 
$L$, 
with reduction parameter 
$\alpha = 3/4$. 
Then, if 
$Y$ 
denotes the basis matrix of 
$L$ 
for 
$\beta$, 
Proposition \ref{basis matrix} says that there exists an 
$(n + 1) \times (n + 1)$ 
matrix 
$C$ 
with integer entries such that 
$Y = CX$. 
That is, there exist 
$c_{j,k} \in \mathbb{Z}$ 
for 
$0 \leq j,k \leq n$ 
such that
\[ Y = 
\begin{pmatrix}
c_{0,0} & c_{0,1} & \cdots & c_{0,n} \\
c_{1,0} & c_{1,1} & \cdots & c_{1,n} \\
\vdots & \vdots & \ddots & \vdots \\
c_{n,0} & c_{n,1} & \cdots & c_{n,n} \\
\end{pmatrix}
\cdot
\begin{pmatrix}
2^{\frac{-n(n + 1)}{4}}\delta^{n + 1} & x_1 & \cdots & x_n \\
0 & -1 & \cdots & 0 \\
\vdots & \vdots & \ddots & \vdots \\
0 & 0 & \cdots & -1
\end{pmatrix}
. \]
In particular,
\[ {\bf y}_1 = \left(c_{0,0}2^{\frac{-n(n + 1)}{4}}\delta^{n + 1}, c_{0,0}x_2 - c_{0,1}, \dots, c_{0,0}x_n - c_{0,n} \right). \]
Put 
$b = c_{0,0}$, 
and for 
$j = 1, \dots, n$, 
put 
$a_j = c_{0,j}$. 
Then,
It follows from the third inequality (i.e., from part (iii) in Proposition \ref{proposition for LLL applied to simultaneous Diophantine approximation}) that
\begin{equation} \label{inequality involving delta} 
\lvert {\bf y}_1 \rvert \leq 2^{\frac{n}{4}}d(L)^{\frac{1}{n + 1}} = \delta < 1.
\end{equation}
Note that if 
$b = 0$, 
then 
$\lvert {\bf y}_1 \rvert \geq 1$, 
where
\[ {\bf y}_1 = \left(0, -a_1, -a_2, \dots, -a_n\right); \]
but that contradicts (\ref{inequality involving delta}). Upon replacing 
${\bf y}_1$ by $-{\bf y}_1$, 
it can be assumed without loss of generality that 
$b \geq 1$. 
Therefore, since the length of 
${\bf y}_1$ is greater than or equal to any of the components of 
${\bf y}_1$, 
\[ 2^{\frac{-n(n + 1)}{4}}\delta^{n + 1}b \leq \lvert {\bf y}_1 \rvert \leq \delta; \] 
so that
\[ b \leq 2^{\frac{n(n + 1)}{4}}\delta^{-n} \]
and hence,
\[ \left\lvert x_j - \frac{a_j}{b}\right \rvert < \frac{\delta}{b},\quad \text{for } 1 \leq j \leq n. \]
This completes the main part of the proof of the theorem. Moreover, since the number of arithmetic operations needed by LLL is $O(n^4\log B)$, where $B$ is a constant that is explicitly determined from the rows of $X$ (see \cite[Proposition 1.26]{lenstra1982factoring}), we now have the desired polynomial-time algorithm. 
\end{proof}

\subsubsection{Description of Our Current Implementation} \label{section 3.2.1}
There have been a number of implementations of the LLL algorithm for generating simultaneous Diophantine approximations to a set of real numbers
\begin{equation} \label{numbers to approximate}
x_1, x_2, \dots, x_N,
\end{equation}
aimed at generating approximations with bounded Dirichlet coefficient (see, e.g., \cite[Definition 1.2, p. 168]{bosma2013finding}) and prescribed quality; see, e.g., \cite{heine1868allgemeine}, \cite{lagarias1985the}, and \cite{bosma2013finding}. The current implementation is focused on computing sequences of increasingly good, simultaneous Diophantine approximations.

Any practical implementation of the LLL algorithm uses rational numbers. For example, in an iterative version of the LLL algorithm from \cite{bosma2013finding}, which finds higher-dimensional simultaneous Diophantine approximations (see \cite[Chapter 2, Theorem 1E]{schmidt1980diophantine}), all of the irrational numbers in their implementation are approximated by rational numbers with denominator 
$2^M$,
for some 
$M \in \mathbb{Z}$ 
(i.e., by dyadic numbers). Our implementation follows the one in \cite[Section 9.2]{bremner2011lattice} and uses the continued fraction process (see, e.g., \cite[Chapter 1]{schmidt1980diophantine}), which generates, for any real number 
$\alpha$, 
an infinite sequence of ``reduced fractions''
\[ \frac{a_1}{b_1}, \frac{a_2}{b_2}, \frac{a_3}{b_3}, \dots \]
that approximate $\alpha$. Each rational number 
$a_j/b_j$ 
is called the 
$j^{\textup{th}}$ {\it convergent} to 
$\alpha$, 
and it is a well-known fact that every convergent 
$a_j/b_j$ 
to 
$\alpha$ 
satisfies the inequality
\[ \left\lvert \alpha - \frac{a_j}{b_j} \right\rvert < \frac{1}{b_j^2}; \]
see, e.g., \cite[Chapter 1]{schmidt1980diophantine}. Therefore, since 
$b_j \to \infty$ 
(see the proof of \cite[Chapter 1, Lemma 4D]{schmidt1980diophantine}), for any real number 
$Q > 1$, 
there exists 
$j \geq 1$ 
such that the ordered pair $Q$, $(b_j,a_j)$ satisfies inequality (\ref{corollary to inequality}) from Lemma \ref{simultaneous approximation}, and thus forms a simultaneous Diophantine approximation to 
$\alpha$; 
see the paragraph preceding the statement of Theorem \ref{approximation theorem}.

Denote the 
$n_j^{\textup{th}}$ 
convergent to the real number 
$x_j$ 
by
\[ \frac{a_{j,n_j}}{b_{j,n_j}},\quad \text{for } j = 1, \dots, N. \]
We start by initializing 
$0 < \delta < 1$ 
(close to 1) and a positive integer 
$n_{\text{steps}}$, 
which together determine the {\it step-size} 
$\Delta\delta = \delta/n_\text{steps}$, 
and then proceed to generate the first convergents
\begin{equation} \label{first convergents}
\frac{a_{1,1}}{b_{1,1}}, \frac{a_{2,1}}{b_{2,1}}, \dots, \frac{a_{N,1}}{b_{N,1}}
\end{equation}
to the real numbers in (\ref{numbers to approximate}) via the continued fraction process. Then, as in the proof of Theorem \ref{LLL theorem}, we take the convergents in (\ref{first convergents}) along with the current value of 
$\delta$, 
and use them both to generate integers 
$b, a_1, \dots, a_n$ 
that satisfy each of the inequalities in (\ref{LLL theorem inequalities}). 

\begin{table}[h]
\centering
\footnotesize
\begin{tabular}[t]{l l l}
$\frac{a_{j,n_j}}{b_{j,n_j}}$ $\circ$ & $\frac{a_j}{b}$ $\bullet$ & $x_1 = \log_2(3)$, $x_2 = \log_2(5)$, $x_3 = \log_2(7)$ $\Diamond$ \\
\midrule
& & $\delta = .10$ \\																															  
$\frac{1054}{665}$ & $\frac{4953}{3125}$ & \includegraphics[width=0.7\textwidth]{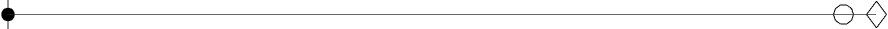}   \\ \\ 
$\frac{1493}{643}$  & $\frac{7256}{3125}$ & \includegraphics[width=0.7\textwidth]{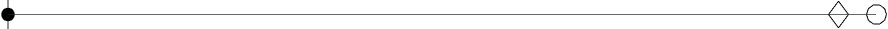}   \\ \\ 
$\frac{6718}{2393}$  & $\frac{8773}{3125}$ & \includegraphics[width=0.7\textwidth]{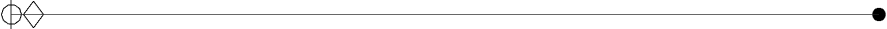}  \\

\midrule
& & $\delta = .01$ \\
$\frac{1054}{665}$ & $\frac{4476254}{2824202}$ & \includegraphics[width=0.7\textwidth]{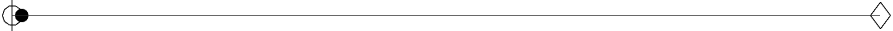}   \\ \\ 
$\frac{1493}{643}$ & $\frac{6557595}{2824202}$ & \includegraphics[width=0.7\textwidth]{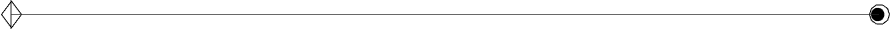}   \\ \\ 
$\frac{6718}{2393}$ & $\frac{7928537}{2824202}$ & \includegraphics[width=0.7\textwidth]{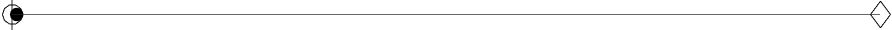}  \\

\midrule	
& & $\delta = .01$ \\
$\frac{50508}{31867}$  & $\frac{163519}{103169}$   & \includegraphics[width=0.7\textwidth]{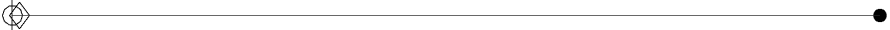} \\ \\ 
$\frac{177797}{76573}$ & $\frac{239551}{103169}$   & \includegraphics[width=0.7\textwidth]{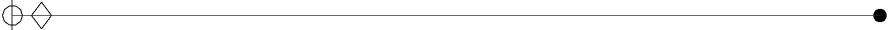} \\ \\ 
$\frac{248027}{88349}$ & $\frac{289632}{103169}$   & \includegraphics[width=0.7\textwidth]{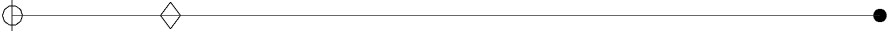} \\ \\
\end{tabular}
\caption{A table showing how the convergents are updated for a fixed $\delta$ to generate two simultaneous Diophantine approximations to the real numbers $\log_2(3)$, $\log_2(5)$, $\log_2(7)$, using the procedure outlined in Section \ref{section 3.2.1}.}
\label{line plots}
\end{table}

For each $1 \leq j \leq N$, define the {\it errors}
\[ E_1(j): = \left\lvert x_j - \frac{a_{j,n_j}}{b_{j,n_j}} \right\rvert\quad \text{ and }\quad E_2(j): = \left\lvert \frac{a_j}{b} - \frac{a_{j,n_j}}{b_{j,n_j}} \right\rvert. \]
We distinguish two cases.

\begin{enumerate}
    \item[] {\it Case 1.}
    If 
    \begin{equation} \label{reasonable}
    E_2(j) \geq 2E_1(j),\quad \text{for } j = 1, \dots, N,
    \end{equation}
    then the simultaneous Diophantine approximation
    \[ Q = \min_{1 \leq j \leq N}\frac{1}{\lvert x_jb - a_j \rvert},\ (b, a_1, \dots, a_N) \] 
    to the real numbers in (\ref{numbers to approximate}) is obtained, and the entire process is repeated after replacing 
    $\delta$ 
    with 
    $\delta - \Delta\delta$. 
    Note that since 
    $\delta < 1$ 
    and 
    $b \geq 1$, we have that  
    \[ \left\lvert x_j - \frac{a_j}{b} \right\rvert < 1,\quad \text{for } j = 1, \dots, N; \]
    so that 
    $Q > 1$. 
    In this case, we say that none of the rational approximations 
    $a_j/b$ 
    can distinguish between the convergent 
    $a_{j,n_j}/b_{j,n_j}$, 
    and the real number $x_j$.
    
    \item[] {\it Case 2.}
    For each 
    $1 \leq j \leq N$ 
    with 
    $E_2(j) < 2E_1(j)$, 
    the convergent 
    $a_{j,n_j}/b_{j,n_j}$ 
    is updated to the next convergent 
    $a_{j,n_{j} + 1}/b_{j,n_{j} + 1}$, 
    which is closer to 
    $x_j$. 
    With the current value of 
    $\delta$ 
    and the updated convergents, a new set of integers 
    $b, a_1, \dots, a_n$ 
    satisfying the inequalities in (\ref{LLL theorem inequalities}) is achieved, and we check whether we are still in {\it Case 2}. Since 
    $\delta$ 
    is fixed, this process of updating convergents and checking will eventually terminate because the denominators of the continued fractions tend to infinity, putting us back in {\it Case 1}, and therefore yielding another simultaneous Diophantine approximation; see Table \ref{line plots}.
\end{enumerate}

\begin{remark}
We choose to use the continued fraction process in our implementation because it is very efficient. Specifically, each convergent is best possible,\footnote{A rational number $a/b$ is a {\it best possible approximation} of a real number $\alpha$ if $\lvert \alpha - a/b \rvert$ does not decrease if $a/b$ is replaced by another rational number written in reduced form and with a smaller denominator.} and in principle, the continued fraction process is identical to the Euclidean algorithm. We also note that our implementation extends the one from \cite{bremner2011lattice} because it finds a sequence of meaningful simultaneous Diophantine approximations, in the sense described in Case 1 and Case 2 above. While it is true that the rational implementation in \cite{bosma2013finding} also finds meaningful approximations, it does not use the continued fraction process.
\end{remark}

\section{The Quasiperiodic Patterns in the Nonlattice Case} \label{section 4}
Using the LSA algorithm, together with our implementation of the LLL algorithm for simultaneous Diophantine approximation (see Section \ref{section 3.2.1}), and the multiprecision polynomial solver MPSolve, which is due to D. A. Bini, G. Fiorentino and L. Robol \cite{bini2000design,bini2014solving}, we study examples of nonlattice Dirichlet polynomials previously studied in, e.g., \cite{lapidus2003complex,lapidus2012fractal}, as well as several new examples, all aimed at illustrating the discussion in the first paragraph, following the statement of Theorem \ref{approximation theorem}, which describes how the quasiperiodic patterns from the complex roots of nonlattice Dirichlet polynomials begin to emerge. 

The visual exploration of the quasiperiodic patterns in this section extend those from the previous works by the two authors of \cite{lapidus2000fractal,lapidus2003complex,lapidus2006fractal,lapidus2012fractal}; see, especially, \cite[Figure 3.6, p.87]{lapidus2012fractal}, \cite[Figure 9, p.62]{lapidus2003complex}, and also \cite[Figure 3.2, p. 71 ]{lapidus2012fractal}, along with the associated examples, which show the roots of several lattice string approximations to nonlattice Dirichlet polynomials in sequence, and illustrate the emergence of a quasiperiodic pattern. Our approach starts by taking the best lattice string approximation $f_q(s)$ that we can compute to a nonlattice Dirichlet polynomial 
$f(s)$, 
where, by ``best'', we mean having the largest 
$\varepsilon$-region of stability. 
We then compute the roots of 
$f(s)$, 
using the complex version of Newton's method with the stable roots of 
$f_q(s)$ 
as initial guesses, up to a region large enough to include several periods of stable roots from a few increasingly good lattice string approximations.\footnote{All of the applications of the LSA algorithm to the examples given in Section \ref{section 4} use $\varepsilon = 1/10$. While this approximation error is small enough to separate the roots of $f(s)$ vertically, it does not separate them horizontally. Therefore, it is possible that we miss some roots when applying Newton's method. In any case, the roots which we obtain suffice to show the emergence of a quasiperiodic pattern, which is the central aim of this section.} By plotting the roots of these lattice string approximations against the roots of $f(s)$, 
we show how the roots of 
$f(s)$ 
are near the roots of each approximation for a certain number of periods of that approximation (sometimes, only for a fraction of a period), and how they eventually start to move away from those roots. We also show plots of the best lattice string approximations that we computed, giving an impression of the quasiperiodic pattern on a larger scale; compare with \cite[Figure 3.7, p. 88]{lapidus2012fractal}. We also present an especially crafted nonlattice Dirichlet polynomial whose roots exhibit a global structure which has not been observed before.

If the rank of the additive group 
$G$ 
corresponding to a nonlattice Dirichlet polynomial 
$f(s)$ is equal to two, we can use the continued fraction process alone in order to generate lattice string approximations to 
$f(s)$. 
Otherwise, the rank is more than two, and we turn to our implementation of the LLL algorithm for simultaneous Diophantine approximation. In any case, we plot the roots of a lattice string approximation 
$f_q(s)$ 
corresponding to the simultaneous Diophantine approximation given by the pair 
$Q$, $(q,k_2, \dots, k_N)$ 
by computing the roots of an associated polynomial in 
$\mathbb{C}[x]$ 
with degree 
$k_N$, 
and which is typically sparse; see the discussion surrounding the equation in (\ref{complex poly}) from Section \ref{section 2.2.1}. In Section \ref{section 4.2}, we show both generic and nongeneric nonlattice examples with rank two, and in Section \ref{section 4.3}, we show generic and nongeneric nonlattice examples with rank three or more. In practice, lattice string approximations with at least one period of stable roots correspond to polynomials with very large degree. Consequently, it takes a significant amount of computing power to implement the approach described above. In \cite[Section 3.8]{lapidus2012fractal}, the first two authors of the present paper state: ``The maximal degree 5000 is the limit of computation: It took several hours with our software on a Sun workstation to compute the golden diagram, which involved solving a polynomial equation of degree 4181. However, finding the roots of the polynomial is the most time-consuming part of the computation. Since these polynomials contain only a few monomials, there may exist ways to speed up this part of the computation.'' In the current paper, using MPSolve and the high performance computer ELSA at the third author's institution, the maximal degree 300000k is now our limit. 

In the present paper, we are careful about the combination of scaling ratios and multiplicities from a nonlattice Dirichlet polynomial 
$f(s)$. 
Indeed, if they are not balanced properly, even the best lattice string approximations that we would be able to compute will have very few stable roots. Specifically, there is still the issue of the size of the LSA constant 
$C$ 
in the radius 
$\varepsilon CQ{\bf p}_q$ 
of the 
$\varepsilon$-region of stability, 
and which in some cases can be smaller than the reciprocal of 
$Q$. 
Therefore, we pay special attention to the LSA constant 
$C$ 
and show how to reverse engineer examples with 
$C$ 
being not too small. Hence, not only do we provide a new implementation of the LLL algorithm that enables us to deal with more complicated examples, but we also give a ``starter-kit'' for the interested reader to explore the quasiperiodic patterns of nonlattice Dirichlet polynomials.

\subsection{A Special Class of Nonlattice Dirichlet Polynomials} \label{section 4.1}
\hfill \\
Let 
\[ f(s) = 1 - \sum_{j = 1}^N m_jr_j^s \]
be a nonlattice Dirichlet polynomial, which is associated to the real numbers 
\begin{equation} \label{associated real numbers}
\frac{w_2}{w_1}, \frac{w_3}{w_1}, \dots, \frac{w_N}{w_1},
\end{equation}
explicitly determined from the scaling ratios of 
$f(s)$. 
Setting 
$\alpha_1 = 1$ 
and 
$\alpha_j = w_j/w_1$, 
for all 
$2 \leq j \leq N$, 
we write
\[ f(s) = 1 - \sum_{j = 1}^N m_jr_1^{\alpha_js}. \]
Recall that each simultaneous Diophantine approximation given by the ordered pair 
$Q$, $(q, k_2, \dots, k_N)$ 
to the real numbers 
$\alpha_2, \dots, \alpha_N$, 
such that 
$k_j/q$ 
approximates 
$\alpha_j$, 
for all integers 
$j$ 
such that
$2 \leq j \leq N$, 
determines the lattice string approximation
\[ f_q(s) = 1 - m_1r_1^s - \sum_{j = 2}^N m_j\left(r_j^{k_j/q}\right)^s \]
to 
$f(s)$, 
and that for any approximation error 
$\varepsilon > 0$, 
the radius of the 
$\varepsilon$-region of stability 
$B_{\varepsilon}(q,Q)$ of $f_q(s)$ 
is 
$\varepsilon CQ{\bf p}_q$, 
where
\[ C := \frac{1}{2\pi}\sum_{j = 1}^N \lvert m_j \rvert\left(\frac{\sum_{j = 0}^N \lvert m_j \rvert}{\min\left\{1,\lvert m_N \rvert\right\}}\right)^{\frac{-2w_N}{\min\left\{w_1,w_{N} - w_{N - 1}\right\}}} \]
is the LSA constant of 
$f(s)$ 
from Theorem \ref{approximation theorem}, and 
${\bf p}_q$ 
is the oscillatory period of 
$f_q(s)$. 

Consider the class of nonlattice Dirichlet polynomials of the form
\begin{equation} \label{type 1}
f(s) = 1 - m_1r^s - m_2r^{\alpha_2s} - m_3r^{\alpha_3s} - \cdots - m_{N - 1}r^{\alpha_{N - 1}s} - r^{\alpha_{N}s},
\end{equation}
where 
$1 > r > 0$
and 
$m_1, m_2, \dots, m_{N - 1} > 0$. 
Note that
\[ 1 < \alpha_2 <  \alpha_3 < \dots < \alpha_{N - 1} < \alpha_{N}, \]
and that since 
$f(s)$ 
is nonlattice, 
$\alpha_j$ 
is irrational, for some integer 
$j$ 
such that
$2 \leq j \leq N$. 
All of the examples of nonlattice Dirichlet polynomials shown in \cite{lapidus2000fractal,lapidus2003complex,lapidus2006fractal,lapidus2012fractal} are of the form (\ref{type 1}), except for both the scaled and unscaled versions of the example in \cite[Example 3.55, p. 113]{lapidus2012fractal}. The following new theorem provides infinitely many nonlattice Dirichlet polynomials of the form (\ref{type 1}), with LSA constant arbitrarily close to 
$(32\pi)^{-1}$.    
 
\begin{theorem} \label{type 1 theorem}
Let 
$f(s)$ 
be a rank two nonlattice Dirichlet polynomial of the form (\ref{type 1}). Then, the LSA constant of 
$f(s)$ 
is given by
\begin{equation} \label{formula for C}
C = C\left(\xi(N),\alpha_{N - 1},\alpha_{N}\right) = \frac{\xi(N) + 1}{2\pi}\left(\frac{1}{\left(\xi(N) + 2\right)^2} \right)^{\frac{\alpha_N}{\min\left\{1,\alpha_N - \alpha_{N - 1}\right\}}},
\end{equation}
where 
$\xi(N) := \sum_{j = 1}^{N - 1} \lvert m_j \rvert$.

Moreover, For any 
$\varepsilon > 0$, 
there exists a nonlattice Dirichlet polynomial of the form (\ref{type 1}) with LSA constant $C$ satisfying the inequality
$0 < (32\pi)^{-1} - C < \varepsilon$.  
\end{theorem}
\begin{proof}
The formula for $C$ is found by a direct substitution; see Equation (\ref{LSA constant}) in Theorem \ref{approximation theorem} above. By requiring that $\alpha_N - \alpha_{N - 1} > 1$, we have
\[ \frac{\alpha_N}{\min\left\{1,\alpha_N - \alpha_{N - 1}\right\}} = \alpha_N. \]
Since we can always manufacture a nonlattice Dirichlet polynomial of the form (\ref{type 1}), with $\xi(N)$ arbitrarily close to zero and $\alpha_N$ arbitrarily close to 2 from the right, the theorem follows.
\end{proof}

\subsection{Rank Two Examples} \label{section 4.2}
\hfill \\
Recall from Definition \ref{rank and generic} that the rank of a nonlattice Dirichlet polynomial 
\[ f(s) = 1 - m_1r_1^s - m_2r_2^s - \cdots - m_Nr_N^s \] 
is equal to the number of rationally independent numbers in the set 
\[ \left\{-\log(r_1), -\log(r_2), \dots, -\log(r_N)\right\}. \]
If the rank is equal to two, then exactly one of the associated real numbers in (\ref{associated real numbers}) is irrational, and we can use the continued fraction process to generate lattice string approximations to 
$f(s)$: 

Suppose that only 
$\alpha = w_2/w_1$ 
is irrational. We start by computing a convergent 
$a/b$ 
to 
$\alpha$, 
and then we modify the remaining real numbers (which are all rational) so that they have denominator 
$b$. 
Then, we obtain the simultaneous Diophantine approximation given by the pair 
$Q$, 
$(q, k_2, k_3, \dots k_N)$ 
to the real numbers in (\ref{associated real numbers}), where 
$q = b$, $k_2 = a$, 
\[ \frac{k_j}{q} = \frac{w_j}{w_1},\quad \text{for } j = 3, \dots, N, \]
and 
\[ Q = \min_{2 \leq j \leq N}\frac{w_1}{\left\lvert w_jq - w_1k_j \right\rvert} = \frac{w_1}{\left\lvert w_2b - w_1a \right\rvert}. \]
Finally, we obtain the lattice sting approximation
\[ f_q(s) = 1 - m_1\left(r_1\right)^s - \sum_{j = 2}^{N} m_j\left(r_1^{k_j/q}\right)^s, \] 
with generator 
$r_1^{1/q}$ 
and oscillatory period 
${\bf p}_q = \frac{2\pi q}{\log r_1^{-1}}$. 

Otherwise, the rank is greater than two 
(since $f(s)$ is nonlattice), 
and in that case we use the LLL algorithm in order to generate lattice string approximations; examples with rank greater than two are discussed in Section \ref{section 4.3}.

\begin{table}[ht]
	\centering
	\begin{tabular}{l l l}
	$f(s) = 1 - 2^{-s} - 3^{-s}$ & & $C \approx 0.0008$ \\
	\hline
	\hline
	$Q,\quad (q,k_2)$ & ${\bf p}_q$ & $\varepsilon CQ{\bf p}$ \\
	\hline
	& & \\
	$192530,\ (111202,176251)$ & $1.00 \times 10^{6}$ & $1.60 \times 10^{7}$ \\
	$189140,\ (79335,125743)$ & 719150 & $1.12 \times 10^{7}$ \\
	$95410,\ (31867,50508)$ & 288865 & $2.27 \times 10^{6}$ \\
	$38096.3,\ (15601,24727)$ & 141419 & $445358$ \\
	$15878.2,\ (665,1054)$ & 6028.04 & $7912.17$ \\
	$678.06,\ (306,485)$ & 2773.8 & $155.478 < {\bf p}_{306}$ \\
	$331.94,\ (53,84)$ & 480.43 & $13.18 < {\bf p}_{53}$ \\
	& & \\
	\end{tabular}
	\caption{Simultaneous Diophantine approximations to the real number $\log_2(3)$ associated to the 2-3 polynomial, and specifications for the corresponding lattice string approximations.} \label{2-3 poly data}
\end{table}

For any lattice string approximation 
$f_q(s)$ 
to a nonlattice Dirichlet polynomial 
$f(s)$, 
we define
\[ a = \min\{\Re(s)\colon f_q(s) = 0\}\quad \text{and}\quad b = \max\{\Re(s)\colon f_q(s) = 0\}. \]
This notation will be used in Section \ref{section 4.2.1} just below.

\subsubsection{The 2-3 and the Golden Polynomials} \label{section 4.2.1}
\begin{example} \label{2-3 example}
The 2-3 {\it polynomial}
\[ f(s) = 1 - 2^{-s} - 3^{-s} \]
from, e.g., \cite[Section 2.3.5, pp. 49 -- 50 and Section 3.8, pp. 115, 117]{lapidus2012fractal}, has scaling ratios 
$r_1 = 2^{-1}$, $r_2 = 3^{-1}$ 
and multiplicities 
$m_1 = m_2 = 1$. 
It is also of the form (\ref{type 1}), with 
$N = 2$ and $\alpha_2 = \log_2(3)$. 

\begin{figure}[p]
	\centering
	\begin{subfigure}{0.5\textwidth}
		\centering
  		\includegraphics[width=.80\textwidth]{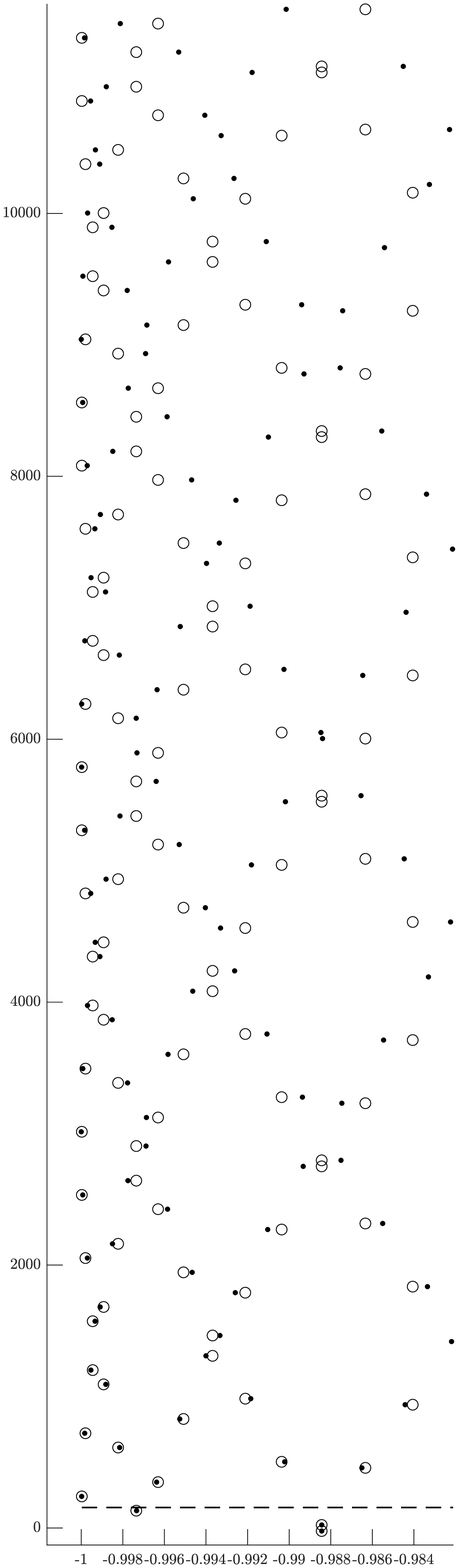} 
		\caption{} 
		\label{lsa_against_2-3_left}
	\end{subfigure}%
	\begin{subfigure}{0.5\textwidth}
  		\centering
  		\includegraphics[width=.80\textwidth]{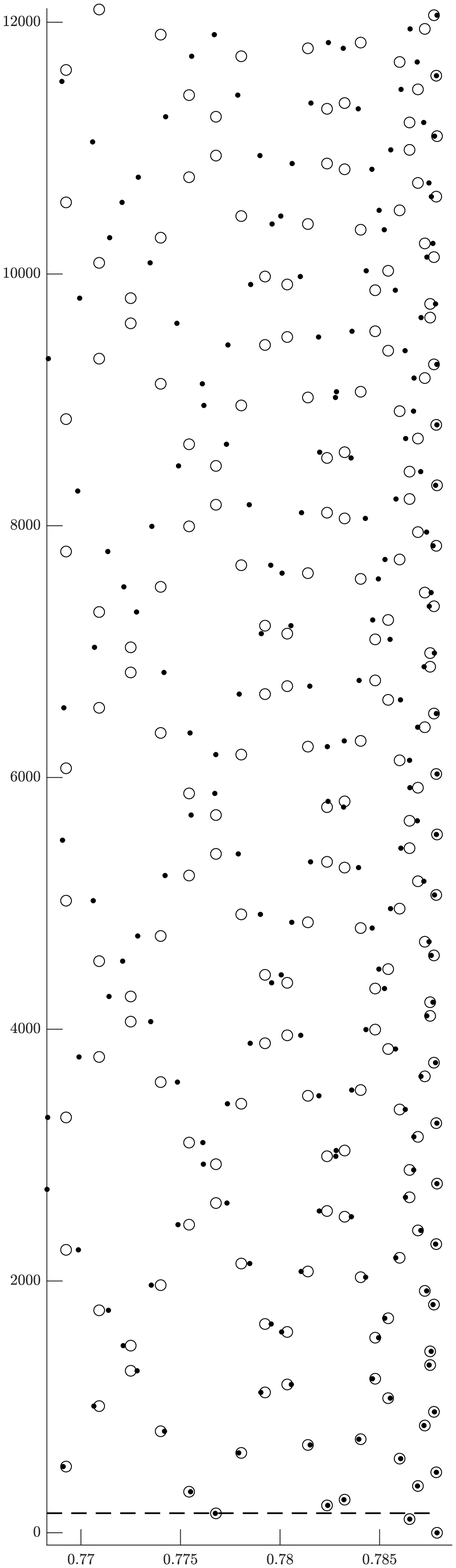} 
		\caption{} 
		\label{lsa_against_2-3_right}
   	\end{subfigure}
   	\caption{Roots of the 2-3 polynomial from Example \ref{2-3 example} near the vertical line (a) $\Re(z) = \min\{\Re(s)\colon f_{306}(s) = 0\}$, (b) $\Re(z) = \max\{\Re(s)\colon f_{306}(s) = 0\}$, moving away from the roots of the lattice string approximation $f_{306}(s)$ (marked with circles).}
\label{lsa_against_2-3}
\end{figure}

\begin{figure}[ht]
	\centering
	\includegraphics[width=.6\textwidth]{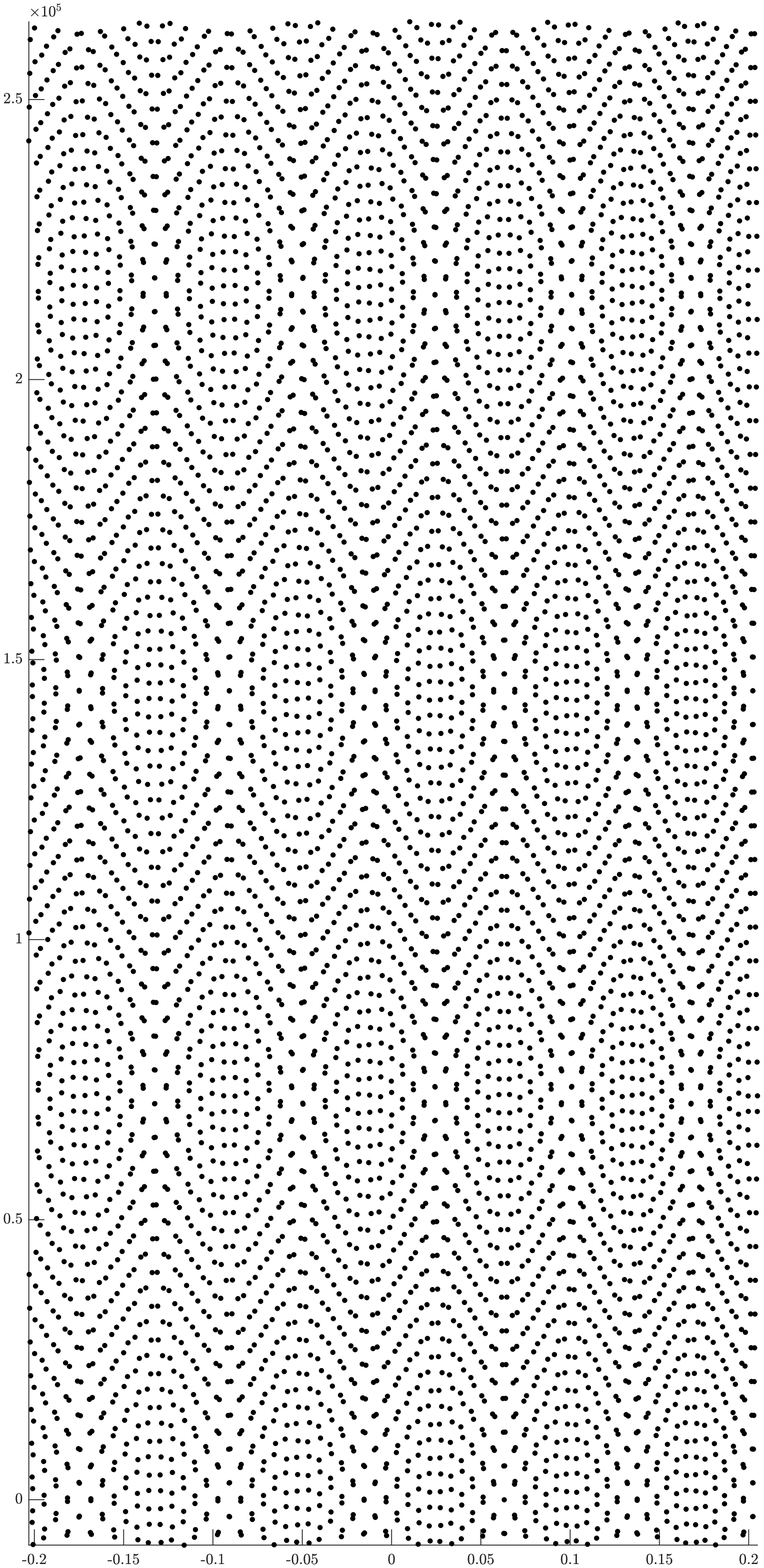} 
	\caption{Stable roots from the lattice string approximation $f_{111202}(s)$ to the 2-3 polynomial from Example \ref{2-3 example}. See Table \ref{2-3 poly data} for the corresponding data.}
   	\label{2-3 large scale plot}
\end{figure}

Table \ref{2-3 poly data} shows data for several lattice string approximations to the 2-3 polynomial. For example, the lattice string approximation
\begin{align*}
f_{15601}(s) &= 1 - 2^{-s} - 2^{-s\frac{24727}{15601}} \\
&= 1 - (2^{-s/15601})^{15601} - (2^{-s/15601})^{24727} 
\end{align*}
has about three periods of stable roots, and it is determined from the simultaneous Diophantine approximation given by the pair $Q = 38096.3$, $(15601,24727)$ to the real number $\log_2(3)$. Note that the best lattice string approximation to the 2-3 polynomial obtained in \cite{lapidus2003complex,lapidus2012fractal} comes from the convergent $485/306$, and does not even have one full period of stable roots. Figure \ref{lsa_against_2-3} shows how the roots of the 2-3 polynomial (marked with dots) are near the roots of $f_{306}(s)$ (marked with circles), and how they eventually start to move away from them. Notice that the point at which the roots of the 2-3 polynomial start to move away from the roots of the lattice string approximation agrees with the theoretical prediction from \cite{lapidus2012fractal}; see Theorem \ref{approximation theorem}. 

We also observe that the roots seem to stay close to the roots of $f_{306}$ for much longer near the extreme vertical lines $\Re(z) = a$ and $\Re(z) = b$. We do not have an explanation for this phenomenon, which occurs in all of the examples discussed in the current paper. In fact, Example \ref{type1_2 example} below shows the roots of a lattice string approximation staying close to a third line. Figure \ref{2-3 large scale plot} displays stable roots of the lattice string approximation $f_{111202}(s)$, giving an impression of the quasiperiodic pattern on a larger scale. We note that large scale plots like the ones in Figure \ref{2-3 large scale plot} were provided for only three out of the ten rank two examples from \cite{lapidus2003complex,lapidus2012fractal}. In the sequel, we will exhibit additional plots of this kind, and arising from far better lattice string approximations. Furthermore, these plots suggest that the patterns from the roots of very good lattice string approximations lie in a wide spectrum, all the way from highly disordered to regular; compare, especially, Figure \ref{2-3 large scale plot} with Figure \ref{3-4-13 large scale plot}.
\end{example}

\begin{example} \label{golden example}
The {\it golden polynomial}
\[ f(s) = 1 - 2^{-s} - 2^{-\phi s} \] 
from, e.g., \cite[Section 2.3.5, pp. 49, 51 -- 52, 53, and Section 3.6, pp.104 -- 106]{lapidus2012fractal}, has scaling ratios $r_1 = 2^{-1}$, $r_2 = 2^{-\phi}$ and multiplicities $m_1 = m_2 = 1$, where $\phi = \frac{1 + \sqrt{5}}{2}$ denotes the golden ratio. 

\begin{figure}[ht]
	\centering
	\includegraphics[width=.6\textwidth]{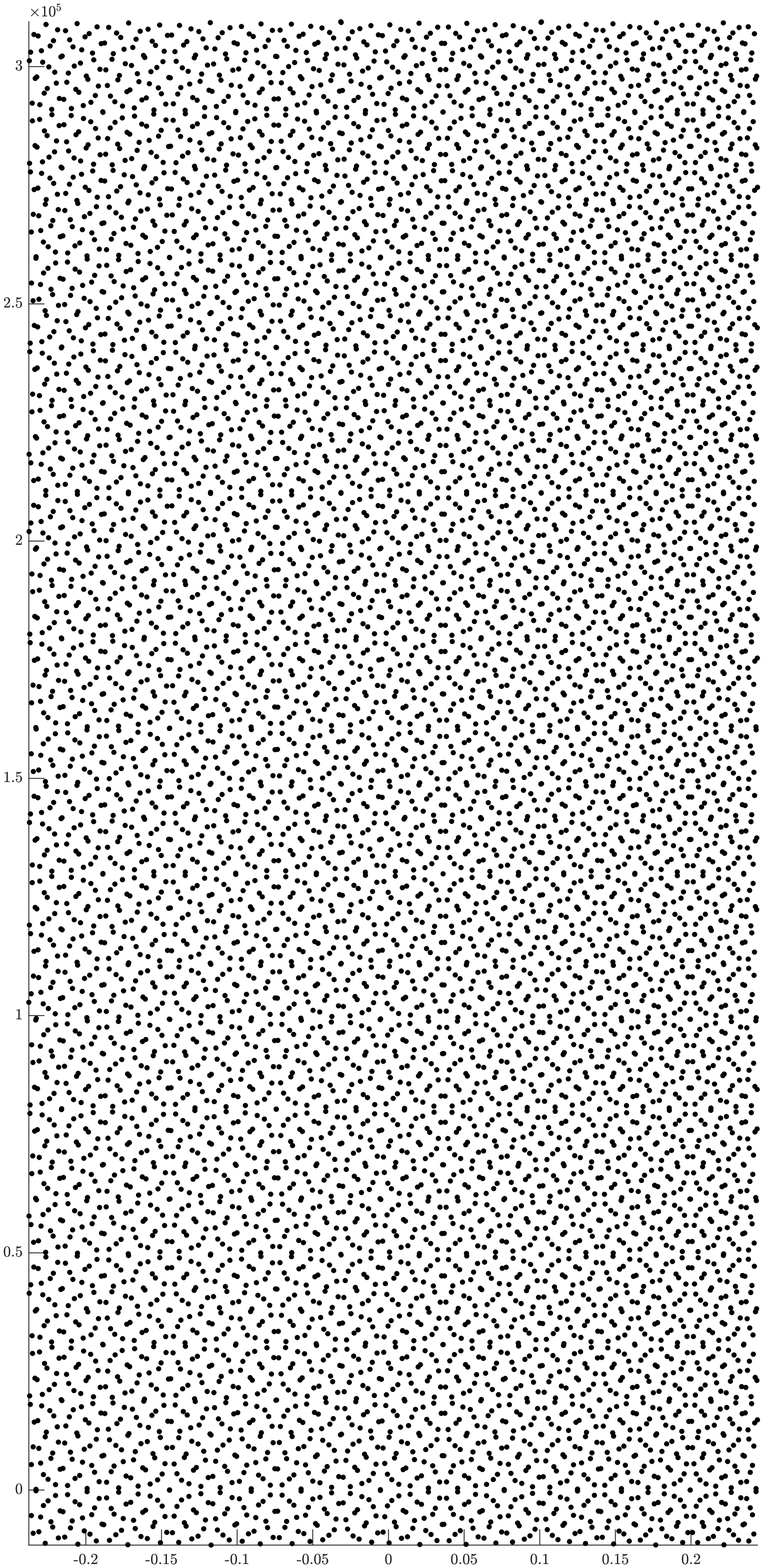} 
	\caption{Stable roots from the lattice string approximation $f_{121393}(s)$ to the golden polynomial from Example \ref{golden example}. See Table \ref{golden poly data} for the corresponding data.}
   	\label{golden large scale plot}
\end{figure}

\begin{table}[ht]
	\centering
	\begin{tabular}{l l l}
	$f(s) = 1 - 2^{-s} - 2^{-\phi s}$ & & $C \approx 0.001$ \\
	\hline
	\hline
	$Q,\quad (q,k_2)$ & ${\bf p}_q$ & $\varepsilon CQ{\bf p}$ \\
	\hline
	& & \\
	$271444,\ (121393,196418)$ & $1.10 \times 10^6$ & $3.01 \times 10^{7}$ \\
	$39603,\ (17711,28657)$ & 160545 & $642594$ \\
	$5778,\ (2584,4181)$ & 23423.2 & $13678.4 < {\bf p}_{2584}$ \\
	$3571,\ (1597,2584)$ & 14476.4 & $5224.68 < {\bf p}_{1597}$ \\
	$2207,\ (987,1597)$ & 8946.88 & $1995.65 < {\bf p}_{987}$ \\
	$1364,\ (610,987)$ & 5529.48 & $762.27 < {\bf p}_{610}$ \\
	& & \\
	\end{tabular}
	\caption{Simultaneous Diophantine approximations to the golden ratio $\phi$ associated to the golden polynomial, and specifications for the corresponding lattice string approximations.} \label{golden poly data}
\end{table}
\begin{figure}[p]
	\centering
	\begin{subfigure}{0.5\textwidth}
		\centering
  		\includegraphics[width=.80\textwidth]{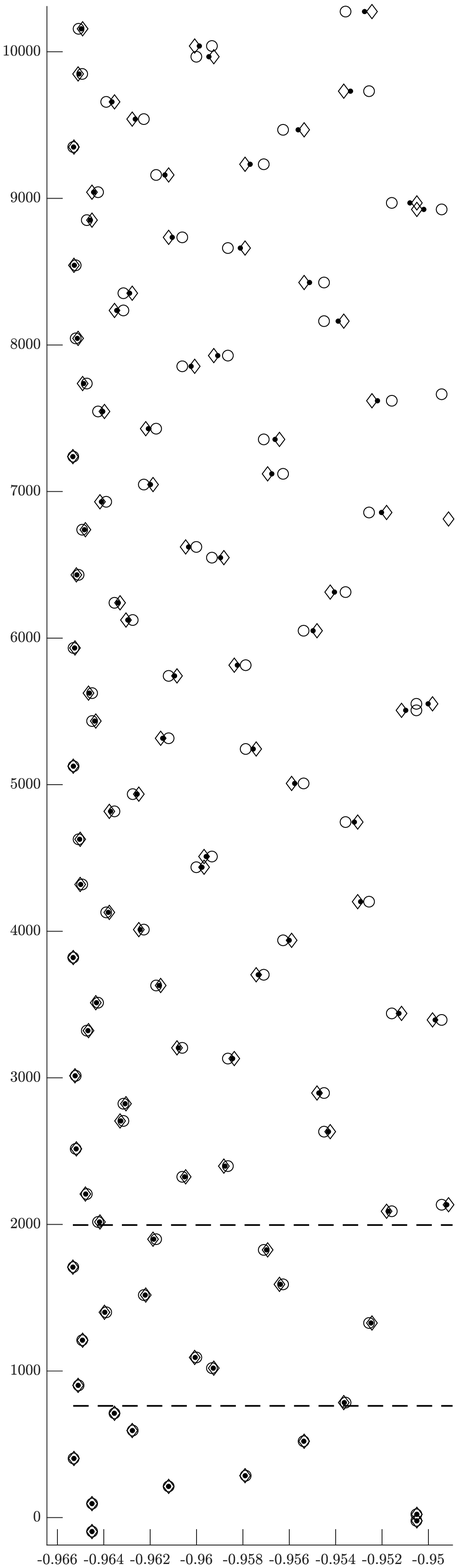} 
		\caption{} 
		\label{lsa_against_golden_left}
	\end{subfigure}%
	\begin{subfigure}{0.5\textwidth}
  		\centering
  		\includegraphics[width=.80\textwidth]{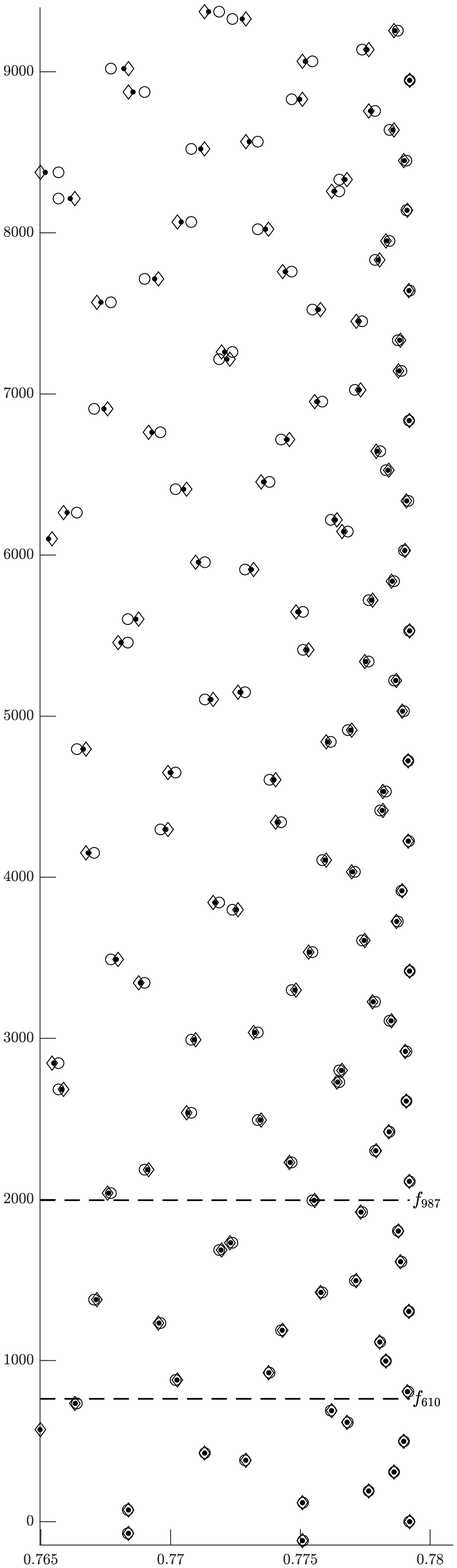} 
		\caption{} 
		\label{lsa_against_golden_right}
   	\end{subfigure}
   	\caption{Roots of the golden polynomial from Example \ref{golden example} near the vertical line (a) $\Re(z) = a$, (b) $\Re(z) = b$ (marked with dots), moving away from the roots of $f_{610}(s)$, $f_{987}(s)$ (marked with circles and diamonds, respectively).}
\label{lsa_against_golden}
\end{figure}

It is also of the form (\ref{type 1}), with $N = 2$ and $\alpha_2 = \phi$; see also the golden$+$ polynomial in, e.g., \cite[Section 3.2.2, pp. 73 -- 74]{lapidus2012fractal}. We compute lattice string approximations to the golden polynomial by generating convergents to the golden ratio. One way to do this, aside from using the continued fraction process, is to take the Fibonacci numbers
\[ 1, 1, 2, 3, 5, 8, 13, 34, 55, \dots \]
and divide each one, starting with the second, by the preceding number.

Table \ref{golden poly data} shows data for several lattice string approximations to the golden polynomial. Notice that we obtain fewer lattice string approximations to the golden polynomial with at least one period of stable roots than we did for the 2-3 polynomial in Example \ref{2-3 example}. On the other hand, our best lattice string approximation to the golden polynomial has roots that are stable for about twenty seven periods, whereas for the 2-3 polynomial studied in Example \ref{2-3 example}, our best one has roots that are stable for about sixteen periods. Figure \ref{lsa_against_golden} shows how the roots of the golden polynomial (marked with dots) are near the roots of the lattice string approximations $f_{610}(s)$ (marked with circles) and $f_{987}(s)$ (marked with diamonds), and how they eventually start to move away from them. We see again that the point at which the roots of the golden polynomial start to move away from the roots of $f_{610}(s)$ and $f_{987}(s)$ is consistent with the theoretical prediction. Figure \ref{golden large scale plot} shows stable roots of the lattice string approximation $f_{121393}(s)$, giving an impression of the quasiperiodic pattern on a larger scale. Observe how the roots constitute a slightly less regular pattern than the one in Figure \ref{2-3 large scale plot}. Also, we note that the best lattice string approximation to the golden polynomial in \cite{lapidus2003complex,lapidus2012fractal} was determined from the convergent $4181/2584$ to the golden ratio $\phi$.
\end{example}

\begin{example} \label{type1_2 example}
Consider the nongeneric nonlattice Dirichlet polynomial
\[ f(s) = 1 - (2^{-s} + 3^{-s} + 4^{-s})\cdot10^{-1} - 6^{-s}. \]
This example was inspired by the 2-3-4-6 polynomial
\[ g(s) = 1 - 2^{-s} - 3^{-s} - 4^{-s} - 6^{-s} \]
from \cite[Section 3.7.1, pp. 113, 115 -- 116]{lapidus2012fractal}, which is of the form (\ref{type 1}), with $N = 4$, $r = 2^{-1}$, $m_1 = m_2 = m_3 = 1$, and $\alpha_2 = \log_2(3)$, $\alpha_3 = 2$, $\alpha_4 = \alpha_2 + 1$. 
The LSA constant of $g(s)$ is of the order of $10^{-7}$, rendering the quasiperiodic pattern of its roots difficult to explore via the LSA algorithm.
\begin{table}[ht]
	\centering
	\begin{tabular}{l l l}
	$f(s) = 1 - (2^{-s} - 3^{-s} - 4^{-s})\cdot10^{-1} - 6^{-s}$ & & $C \approx 0.0001$ \\
	\hline
	\hline
	$Q,\quad (q,k_2)$ & ${\bf p}_q$ & $\varepsilon CQ{\bf p}$ \\
	\hline
	& & \\
	$192530,\ (111202,176251)$ & $1.00 \times 10^{6}$ & $2.55 \times 10^{6}$ \\
	$189140,\ (79335,125743)$ & 719150 & $1.78 \times 10^{6}$ \\
	$95410,\ (31867,50508)$ & 288865 & $362315$ \\
	$38096.3,\ (15601,24727)$ & 141419 & $70825 < {\bf p}_{15601}$ \\
	$15878.2,\ (665,1054)$ & 6028.04 & $1258.27 < {\bf p}_{665}$ \\
	$678.06,\ (306,485)$ & 2773.8 & $24.72 < {\bf p}_{306}$ \\
	$331.94,\ (53,84)$ & 480.43 & $2.09 < {\bf p}_{53}$ \\
	& & \\
	\end{tabular}
	\caption{This table shows simultaneous Diophantine approximations to the real number $\log_2(3)$ associated to the nonlattice Dirichlet polynomial from Example \ref{type1_2 example}, along with some specifications for the corresponding lattice string approximations.} \label{type1_2 poly data}
\end{table}

\begin{figure}[p] 
s\centering
\includegraphics[width=.65\textwidth]{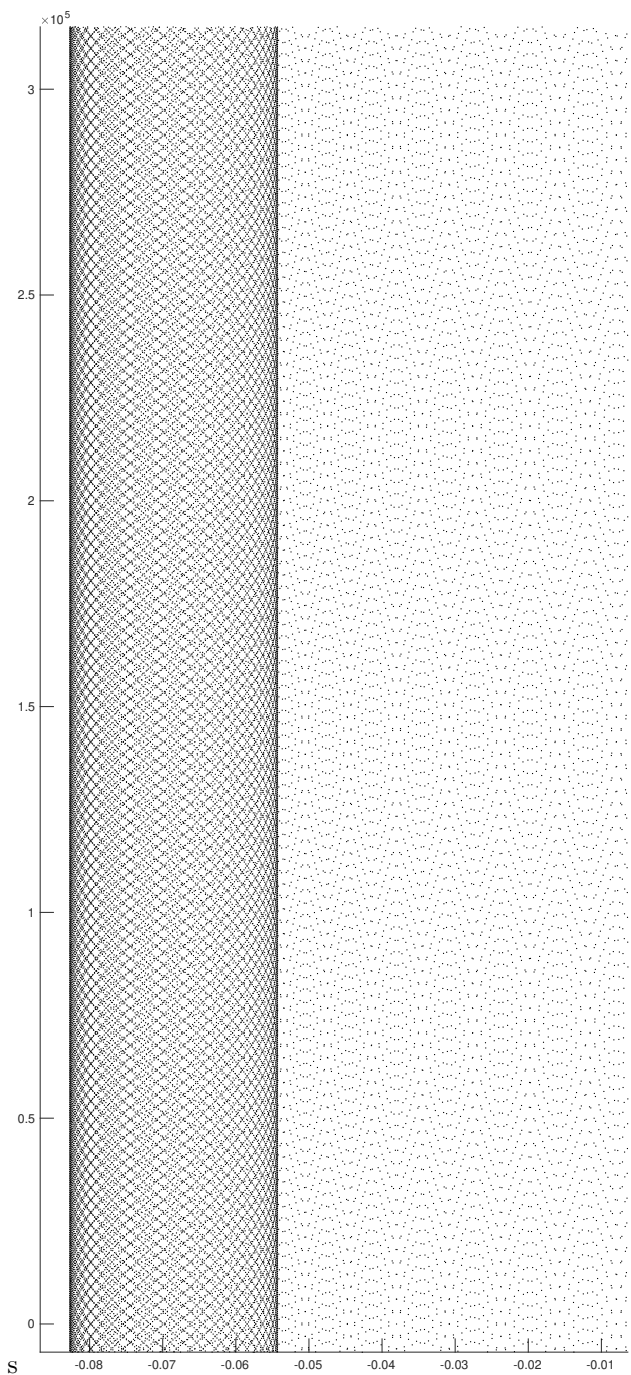} 
\caption{Stable roots from the lattice string approximations $f_{111202}(s)$ to the nongeneric nonlattice Dirichlet polynomial from Example \ref{type1_2 example}. See Table \ref{type1_2 poly data} for the corresponding data.}
\label{type1_2 large scale plot}
\end{figure}

\begin{figure}[ht]
	\centering
	\begin{subfigure}{0.3\textwidth}
		\centering
  		\includegraphics[width=.7\textwidth]{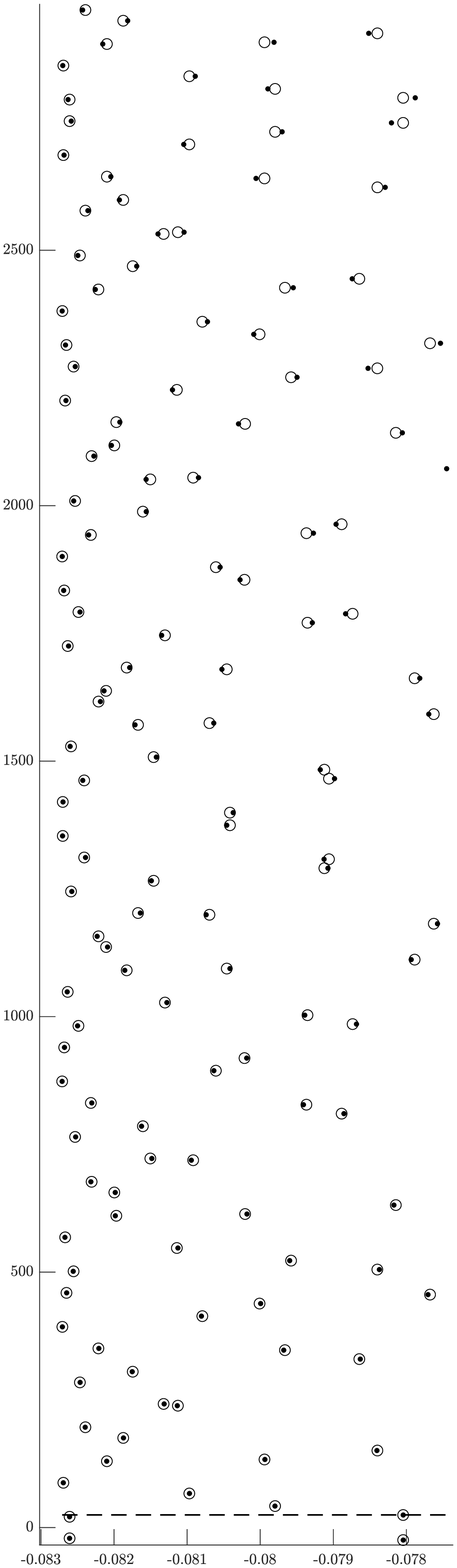} 
		\caption{} 
		\label{lsa_against_type1_2_left}
	\end{subfigure}%
	\begin{subfigure}{0.3\textwidth}
  		\centering
  		\includegraphics[width=.7\textwidth]{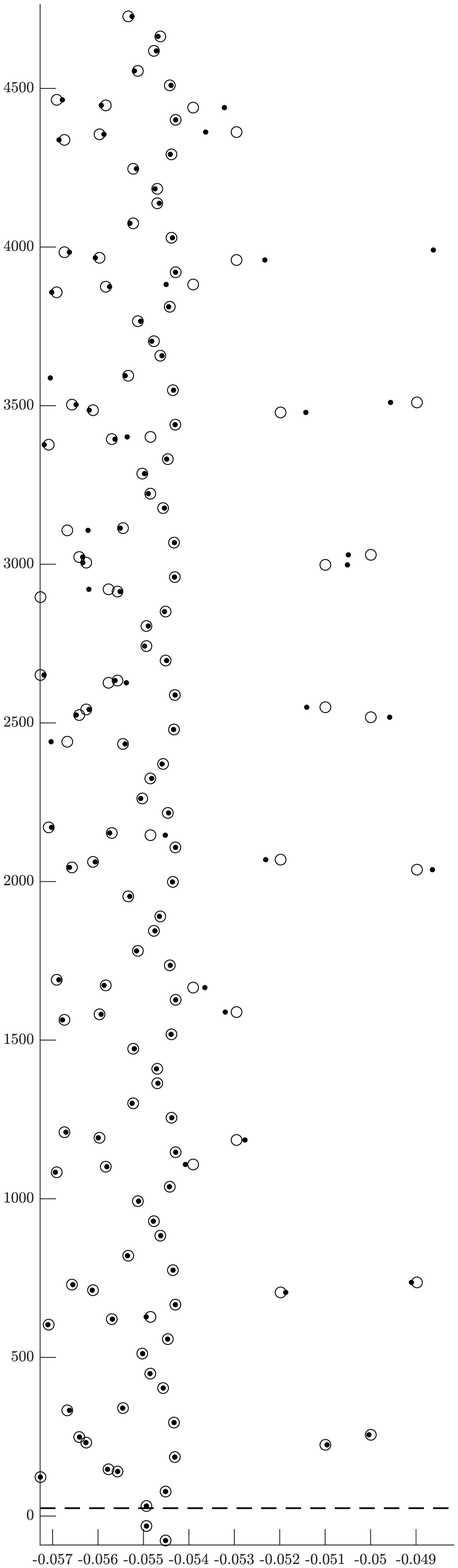} 
		\caption{} 
		\label{lsa_against_type1_2_middle}
   	\end{subfigure}%
	\begin{subfigure}{0.3\textwidth}
  		\centering
  		\includegraphics[width=.7\textwidth]{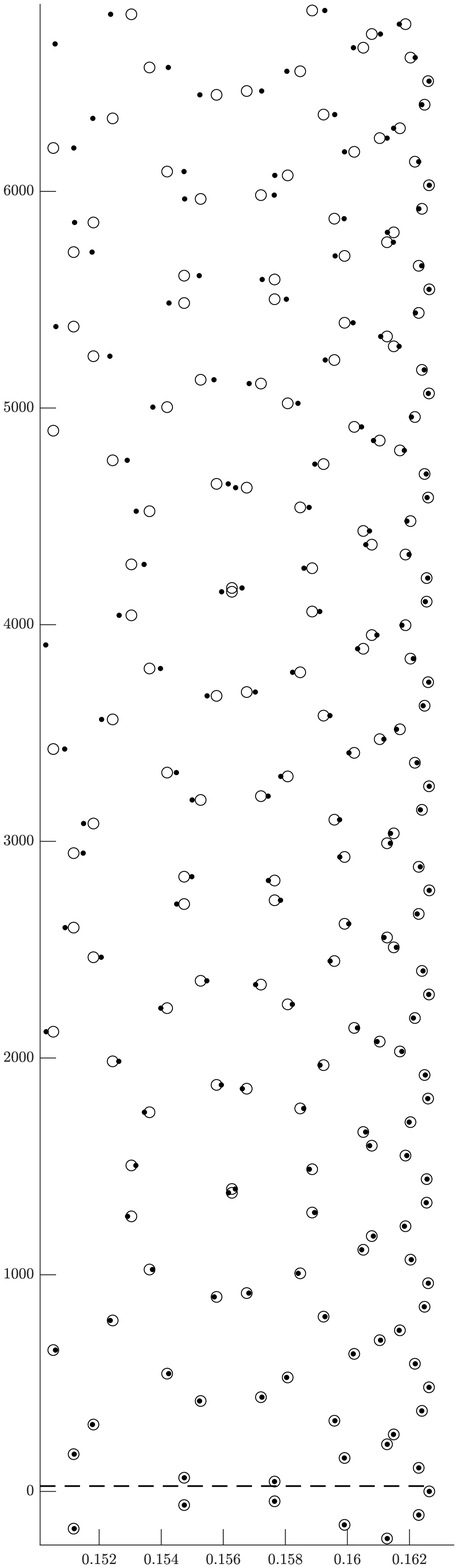} 
		\caption{} 
		\label{lsa_against_type1_2_right}
   	\end{subfigure}
   	\caption{Roots of $f(s)$ from Example \ref{type1_2 example} near the vertical line (a) $\Re(z) = a$, (b) $\Re(z) = -0.0543$, (c) $\Re(z) = b$ (marked with dots), moving away from the roots of $f_{306}(s)$ (marked with circles).}
\label{lsa_against_type1_2}
\end{figure}

\begin{table}[ht]
	\centering
	\begin{tabular}{l l l}
	$f(s) = 1 - 2^{-s} - 3^{-s} - 5^{-s} - 7^{-s}$ & & $C \approx 5.2 \times 10^{-9}$ \\
	\hline
	\hline
	$Q,\quad (q,k_2,k_3,k_4)$ & ${\bf p}_q$ & $\varepsilon CQ{\bf p}$ \\
	\hline
	& & \\
	$265.73,\ (103169,163519,239551,289632)$ & $935198$ & $0.13 < {\bf p}_{103169}$ \\
	$75.19,\ (18355,29092,42619,51529)$ & 166383 & $0.006 < {\bf p}_{18355}$ \\
	$39.53,\ (3125,4953,7256,8773)$ & 28327.3 & $0.0005 < {\bf p}_{3125}$ \\
	$22.97,\ (441,699,1024,1238)$ & 3997.54 & $4.8 \times 10^{-5} < {\bf p}_{441}$ \\
	$17.33,\ (171,271,397,480)$ & 1550.07 & $1.40 \times 10^{-5} < {\bf p}_{171}$ \\
	& & \\
	\end{tabular}
	\caption{This table shows simultaneous Diophantine approximations $Q$, $(q,k_2,k_3)$ to the real numbers $\log_2(3)$, $\log_2(5)$ associated to the nonlattice Dirichlet polynomial in Example \ref{2-3-5-7 example}, along with some specifications for their corresponding lattice string approximations.} \label{2-3-5-7 poly data}
\end{table}

By scaling the multiplicities $m_1$, $m_2$, $m_3$ by a factor of $10^{-1}$, we obtain the nongeneric nonlattice Dirichlet polynomial $f(s)$, with LSA constant $C = 0.0001$, calculated by using the formula (\ref{formula for C}) in Theorem \ref{type 1 theorem}. As a result, we obtain a more tractable version of the 2-3-4-6 polynomial. Table \ref{type1_2 poly data} shows data for several lattice string approximations to $f(s)$, and just like for the 2-3 polynomial in Example \ref{2-3 example}, we generate lattice string approximations by approximating $\log_2(3)$. We note, however, that they are not as good as the ones from Table \ref{2-3 poly data} because the LSA constant is smaller. Figure \ref{lsa_against_type1_2} displays the roots of $f(s)$ moving away from the roots of $f_{306}(s)$. We again observe that the roots remain close to the roots of $f_{306}(s)$ near the two extreme vertical lines $\Re(z) = a$ and $\Re(z) = b$, as shown in parts (a) and (c) of Figure \ref{lsa_against_type1_2}, respectively. However, we see in part (b) of Figure \ref{lsa_against_type1_2} some of the roots staying close to a third line, $\Re(z) = -0.0543$. Figure \ref{type1_2 large scale plot} shows one period of stable roots of the lattice string approximation $f_{111202}$, giving an impression of two distinct quasiperiodic patterns on each side of the vertical line $\Re(z) = -0.0543$. Note that while the part to the left is smaller than the part on the right, both parts consist of roughly the same number of roots, which may be inferred by observing that the roots in the left part seem to be more densely distributed than the ones in the right part. 
\end{example}

\subsection{Rank Three or More Examples} \label{section 4.3}
\hfill \\
In this section, we discuss examples with rank greater than two, which make use of our implementation of the LLL algorithm. While the authors of \cite{lapidus2003complex,lapidus2012fractal} did show two examples with rank greater than two, specifically, the 2-3-5 polynomial of rank three from, e.g., \cite[Section 3.6, pp. 108--109]{lapidus2012fractal} and the 2-3-5-7 polynomial of rank four from \cite[Section 3.7.1, pp. 113--114]{lapidus2012fractal}, the visual presentation of the roots was severely limited, due to the absence of tools like the LLL algorithm and MPSolve. Here, we plot the roots of the 2-3-5-7 polynomial, and also the roots of two new examples, all of which are of the form (\ref{type 1}). 

Recall that generating lattice string approximations to nonlattice Dirichlet polynomials with rank greater than two means generating simultaneous Diophantine approximations to two or more real numbers. The naive approach is to build off the process described in the second to last paragraph before Section \ref{section 4.2.1}, and simply generate a convergent to each irrational number and then find a common denominator. However, this process is extremely inefficient.  

\begin{figure}[p]
	\centering
	\begin{subfigure}{0.5\textwidth}
		\centering
  		\includegraphics[width=.80\textwidth]{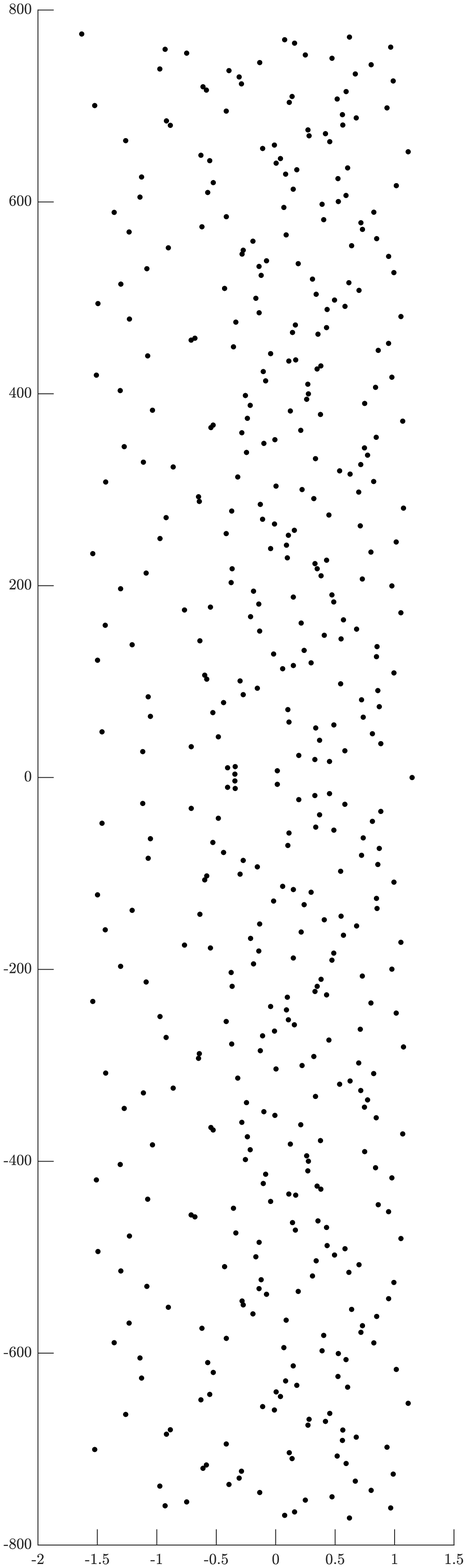} 
		\caption{} 
		\label{2-3-5-7_171}
	\end{subfigure}%
	\begin{subfigure}{0.5\textwidth}
  		\centering
  		\includegraphics[width=.80\textwidth]{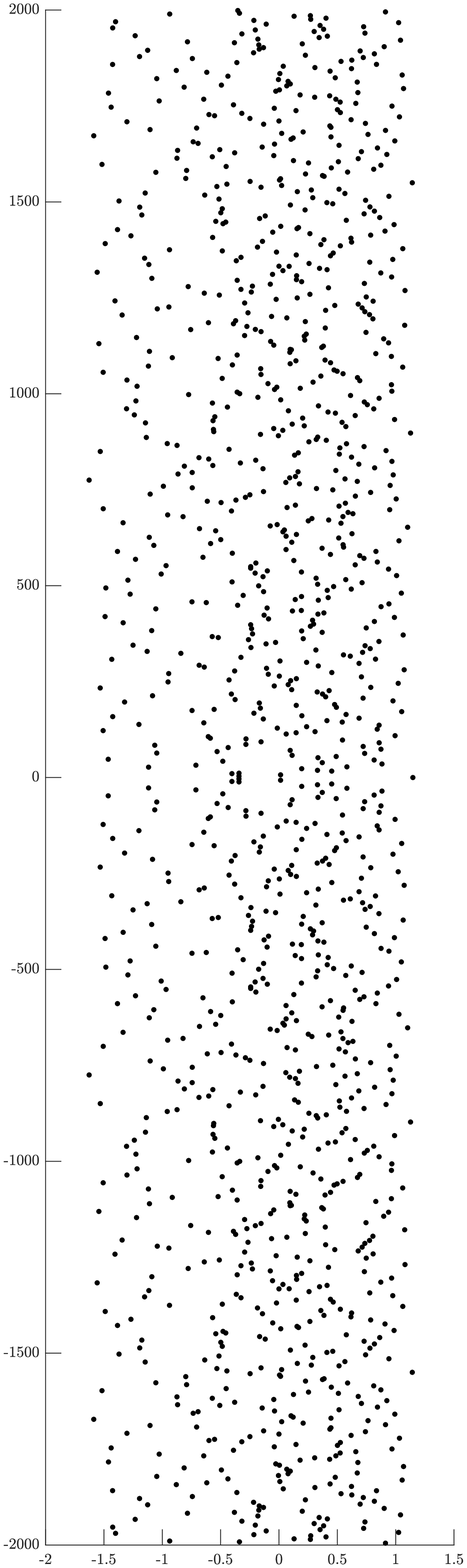} 
		\caption{} 
		\label{2-3-5-7_441}
   	\end{subfigure}
   	\caption{Roots of the lattice string approximations (a) $f_{171}(s)$, (b) $f_{441}(s)$, to the 2-3-5-7 polynomial from Example \ref{2-3-5-7 example}.}
\label{2-3-5-7_171_441}
\end{figure}

\begin{figure}[p]
	\centering
	\begin{subfigure}{0.5\textwidth}
		\centering
  		\includegraphics[width=.80\textwidth]{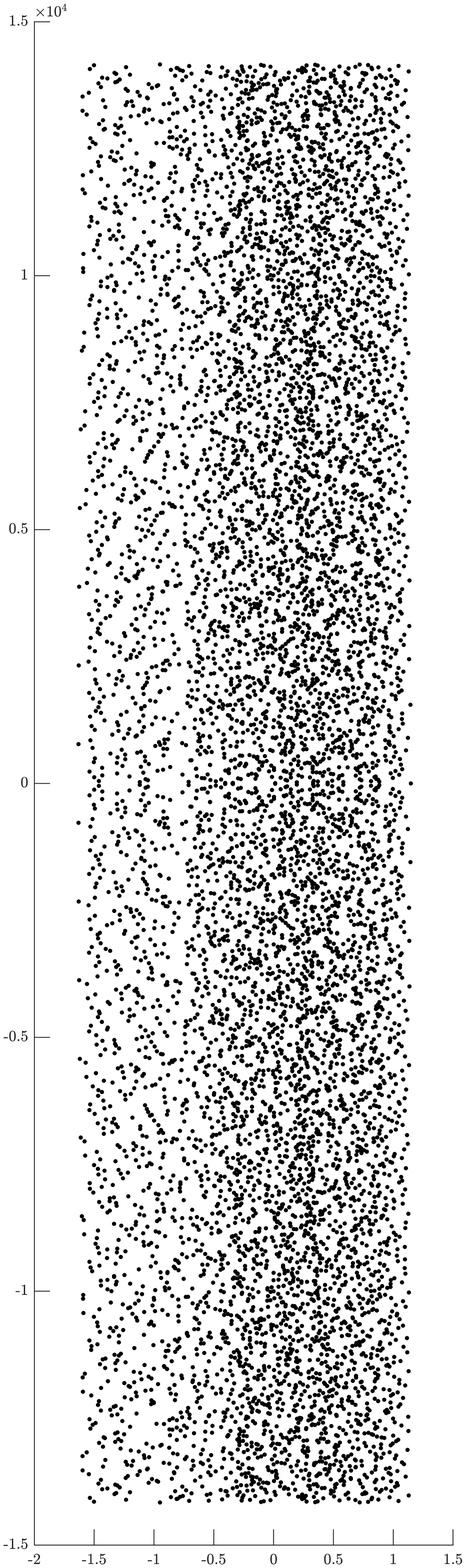} 
		\caption{} 
		\label{l2-3-5-7_3125}
	\end{subfigure}%
	\begin{subfigure}{0.5\textwidth}
  		\centering
  		\includegraphics[width=.80\textwidth]{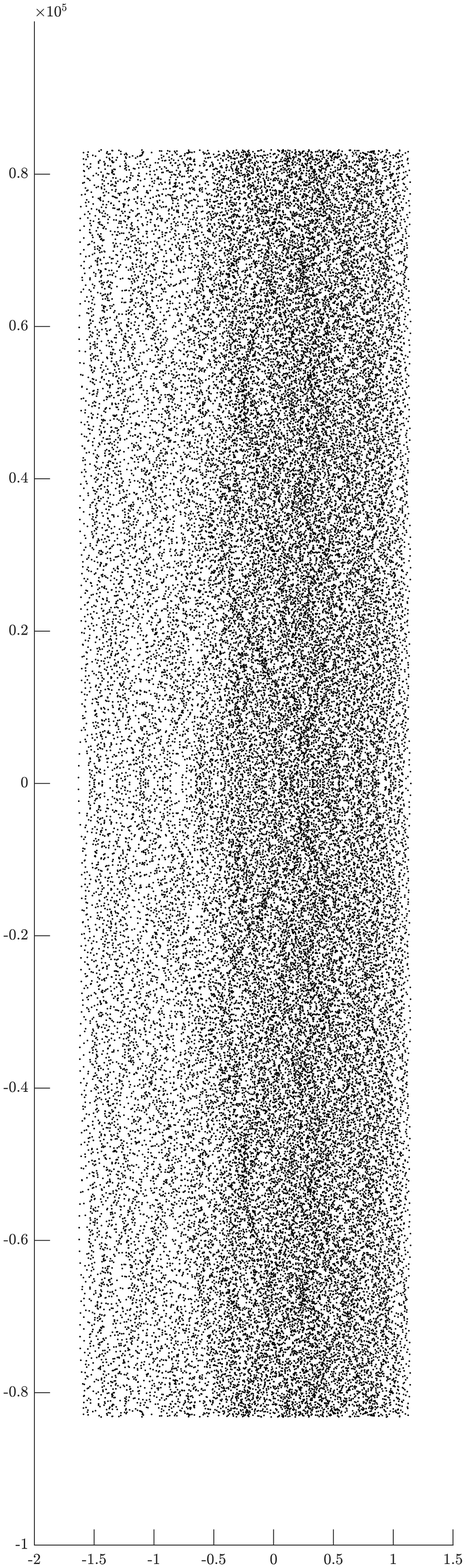} 
		\caption{} 
		\label{2-3-5-7_18355}
   	\end{subfigure}
   	\caption{Roots of the lattice string approximations (a) $f_{3125}(s)$, (b) $f_{18355}(s)$, to the 2-3-5-7 polynomial from Example \ref{2-3-5-7 example}.}
\label{2-3-5-7_3125_18355}
\end{figure}

\begin{example} \label{2-3-5-7 example}
The 2-3-5-7 polynomial
\[ f(s) = 1 - 2^{-s} - 3^{-s} - 5^{-s} - 7^{-s} \]
from \cite[Section 3.7.1, pp. 113 -- 114]{lapidus2012fractal} is a generic nonlattice example with rank four. It is of the form (\ref{type 1}), with LSA constant $C = 5.2 \times 10^{-9}$. Table \ref{2-3-5-7 poly data} shows the data for the lattice string approximations we computed, and we note that with such a small LSA constant, and because the LLL algorithm is not as efficient as the continued fraction process, we cannot even get a single lattice string approximation with roots that are stable for at least one period. Figure \ref{2-3-5-7_171_441} and Figure \ref{2-3-5-7_3125_18355} show the roots of each lattice string approximation from Table \ref{2-3-5-7 poly data}.
\end{example}

\begin{example} \label{3-4-13 example}

\begin{figure}[p] 
\centering
\includegraphics[width=.70\textwidth]{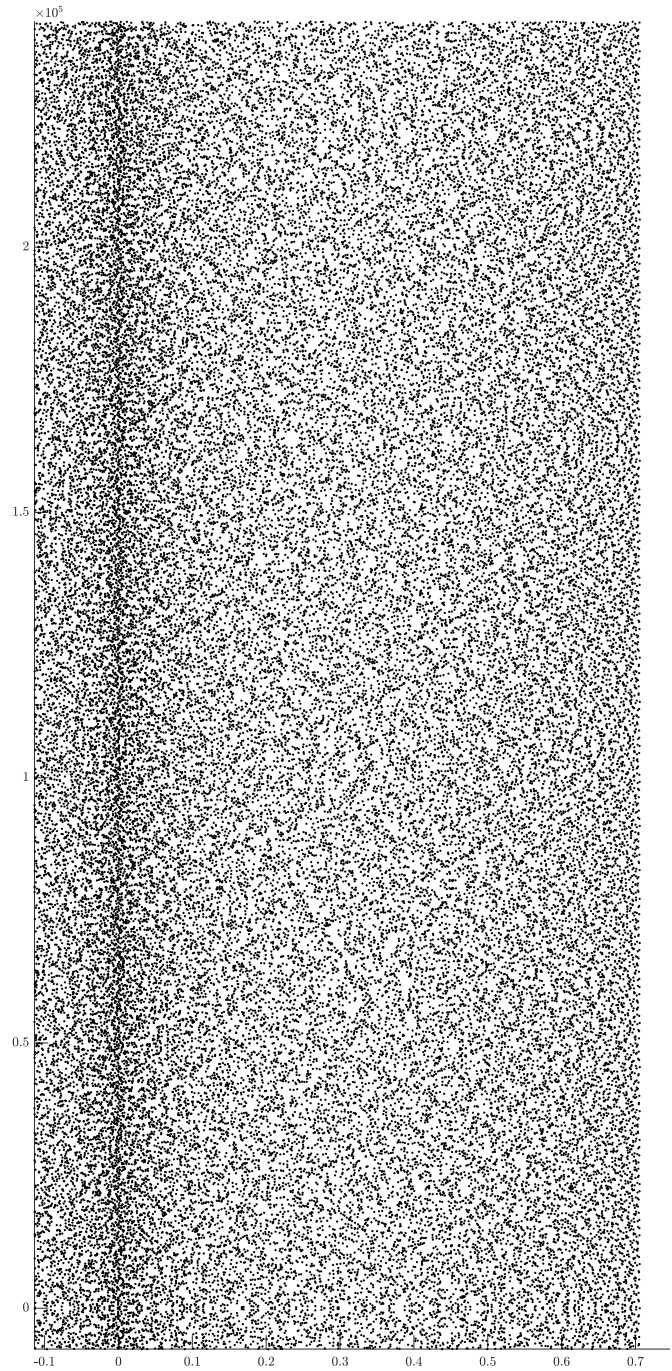} 
\caption{Stable roots from the lattice string approximations $f_{85248}(s)$ to the 3-4-13 polynomial from Example \ref{3-4-13 example}. See Table \ref{3-4-13 poly data} for the corresponding data.}
\label{3-4-13 large scale plot}
\end{figure}

The 3-4-13 polynomial
\[ f(s) = 1 - 3^{-s} - 4^{-s} - 13^{-s} \]
is a generic nonlattice example with rank three. It is of the form (\ref{type 1}), with LSA constant $C = 0.0007$; see Table \ref{3-4-13 poly data} for the data on the lattice string approximations we computed. Figure \ref{3-4-13 large scale plot} displays the roots of the lattice string approximation $f_{85248}(s)$. Observe how irregular the roots are, when compared to the other examples shown in the rest of Section \ref{section 4}. Also, note the dense pocket of roots near the imaginary axis.
\end{example}

\begin{table}[h]
	\centering
	\begin{tabular}{l l l}
	$f(s) = 1 - 3^{-s} - 4^{-s} - 13^{-s}$ & & $C \approx 0.0007$ \\
	\hline
	\hline
	$Q,\quad (q,k_2,k_3)$ & ${\bf p}_q$ & $\varepsilon CQ{\bf p}$ \\
	\hline
	& & \\
	$1111.13,\ (85248,107571,199030)$ & $487551$ & $39943.5 < {\bf p}_{85248}$ \\
	$400.52,\ (7947,10028,18554)$ & 45450.5 & $1342.23 < {\bf p}_{7947}$ \\
	$185.72,\ (4090,5161,9549)$ & 23391.5 & $320.33 < {\bf p}_{4090}$ \\
	$75.38,\ (233,294,544)$ & 1332.57 & $7.4 < {\bf p}_{233}$ \\
	& & \\
	\end{tabular}
	\caption{This table shows simultaneous Diophantine approximations $Q$, $(q,k_2,k_3)$ to the real numbers $\log_3(4)$, $\log_3(13)$ associated to the nonlattice Dirichlet polynomial in Example \ref{3-4-13 example}, along with some specifications for their corresponding lattice string approximations.} 
	\label{3-4-13 poly data}
\end{table}

\begin{figure}[p]
	\centering
	\begin{subfigure}{0.5\textwidth}
  		\centering
  		\includegraphics[width=.80\textwidth]{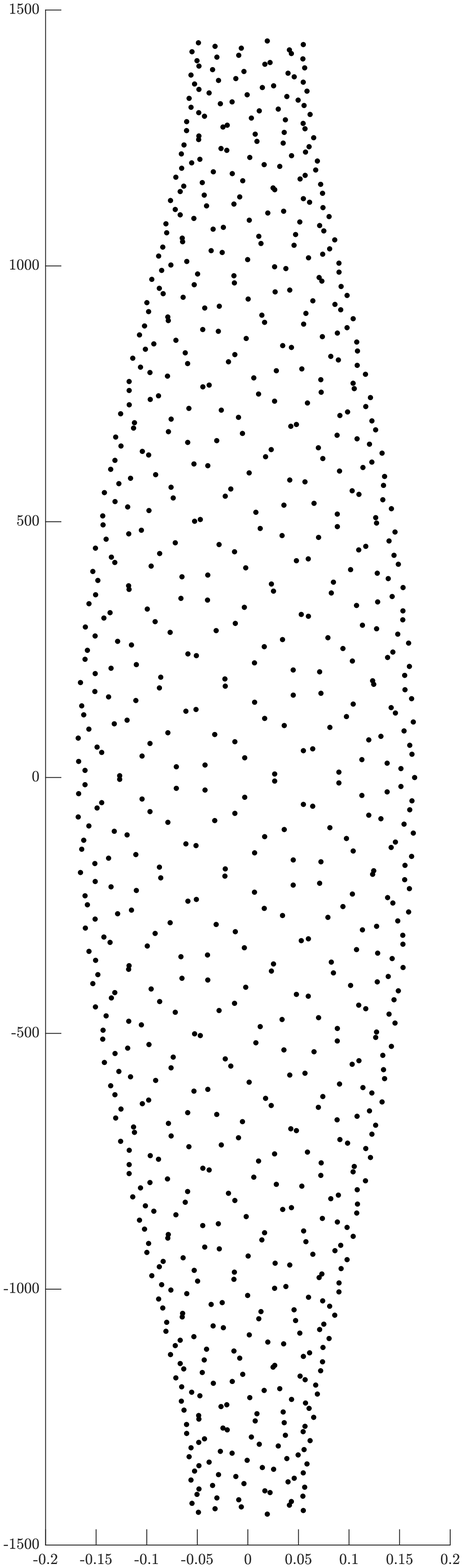} 
		\caption{} 
	\end{subfigure}%
	\begin{subfigure}{0.5\textwidth}
  		\centering
  		\includegraphics[width=.80\textwidth]{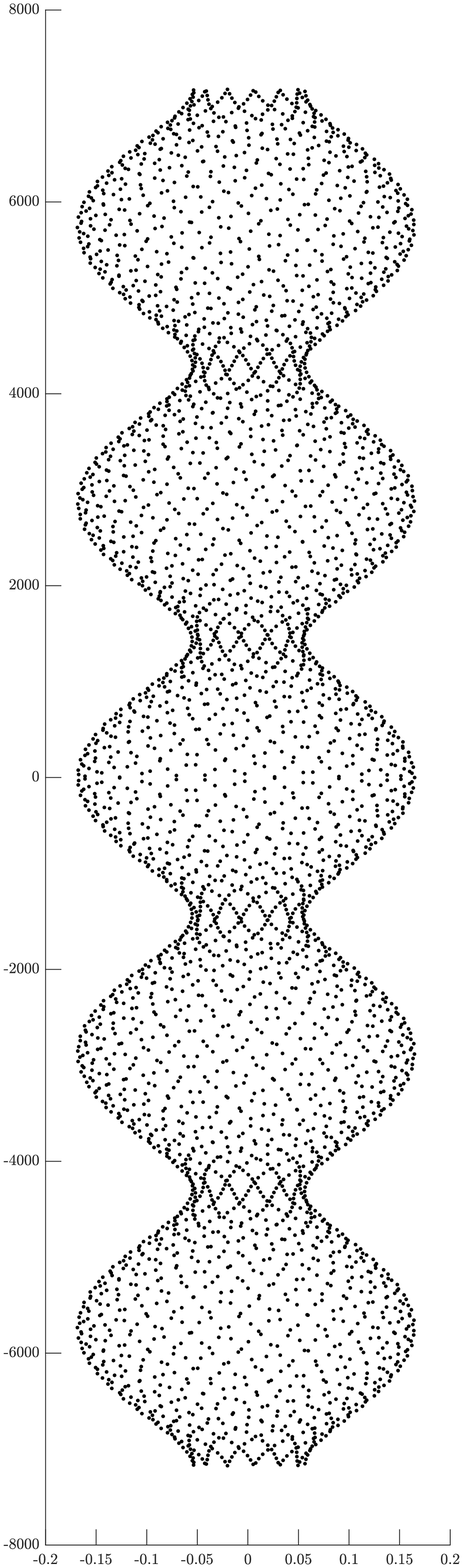} 
		\caption{}
	\end{subfigure}
\caption{Roots of the lattice string approximations (a) $f_{318}(s)$, (b) $f_{1583}(s)$, to $f(s)$ from Example \ref{type1_3 example}.}
\label{type1_3 large scale plot}
\end{figure}

\begin{figure}[p]
	\centering
	\begin{subfigure}{0.5\textwidth}
  		\centering
  		\includegraphics[width=.80\textwidth]{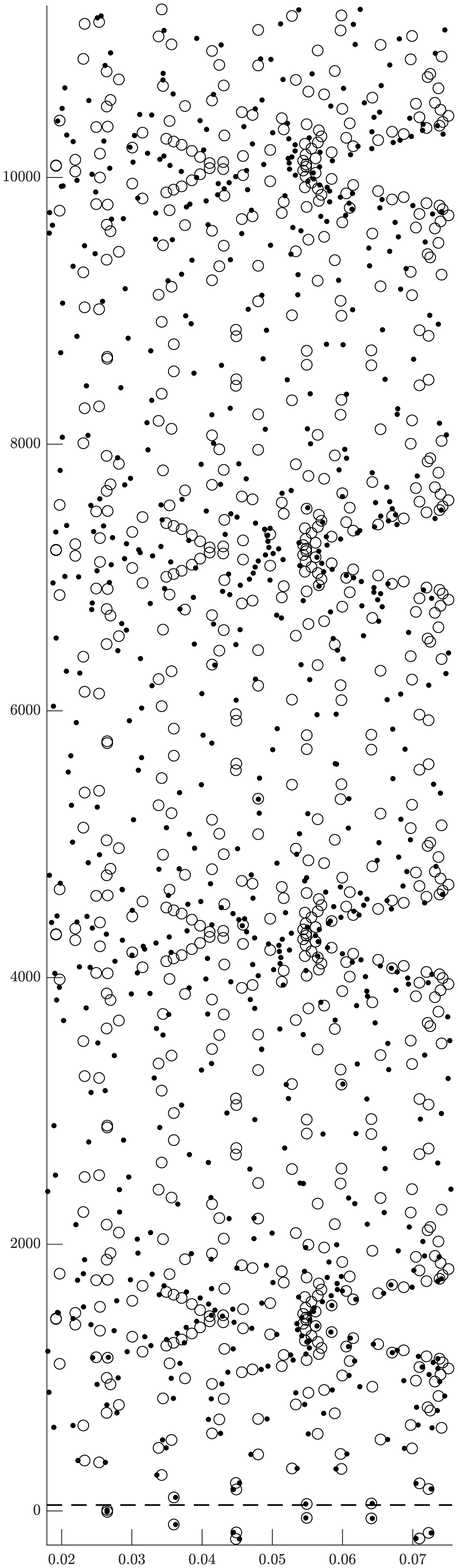} 
		\caption{} 
	\end{subfigure}%
	\begin{subfigure}{0.5\textwidth}
  		\centering
  		\includegraphics[width=.80\textwidth]{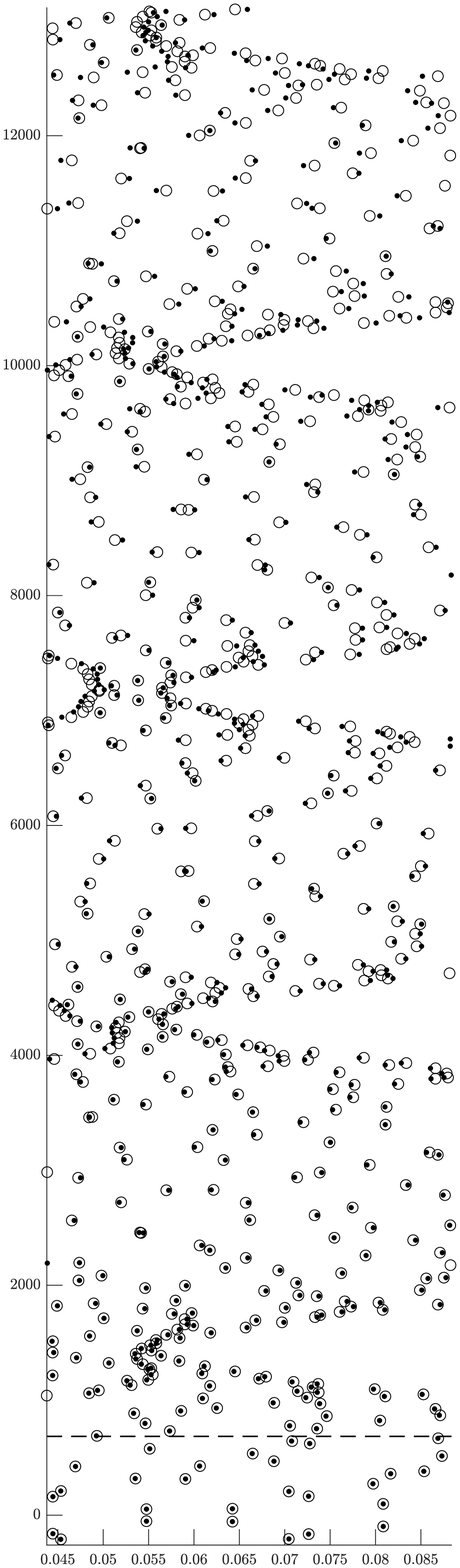} 
		\caption{}
	\end{subfigure}
\caption{Roots of $f(s)$ from Example \ref{type1_3 example} (marked with dots) moving away from the roots of (a) $f_{318}(s)$, (b) $f_{1583}(s)$ (marked with circles). The dotted line indicates the theoretical region of stability.}
\label{lsa_against_type1_3}
\end{figure}

\begin{example} \label{type1_3 example}
Let $a > 0$ be an irrational number, and consider the nonlattice Dirichlet polynomial
\[ f(s) = 1 - m_12^{-s} - m_22^{-\alpha_2s} - m_32^{-\alpha_3s} - 2^{-\alpha_4s}, \]
where $\alpha_2 > 1$ is irrational, $\alpha_3 = \alpha_2 + a$, and $\alpha_4 = \alpha_3 + 1$. 

\begin{table}[ht]
	\centering
	\begin{tabular}{l l l}
	$f(s) = 1 - m_12^{-s} - m_22^{-\alpha_2s} - m_32^{-\alpha_3s} - 2^{-\alpha_4s}$ & & $C \approx 0.002$ \\
	\hline
	\hline
	$Q,\quad (q,k_2,k_3)$ & ${\bf p}_q$ & $\varepsilon CQ{\bf p}$ \\
	\hline
	& & \\
	$997.5,\ (100562,159387,318)$ & $911566$ & $252396 < {\bf p}_{100562}$ \\
	$537.25,\ (46803,74181,148)$ & 424256 & $63267 < {\bf p}_{46803}$ \\
	$283.56,\ (6956,11025,22)$ & 63054.2 & $4962.89 < {\bf p}_{6956}$ \\
	$172.10,\ (1583,2509,5)$ & 14349 & $685.48 < {\bf p}_{1583}$ \\
	$55.32,\ (318,504,1)$ & 2882.58 & $44.26 < {\bf p}_{318}$ \\
	& & \\
	\end{tabular}
	\caption{This table shows simultaneous Diophantine approximations $Q$, $(q,k_2,k_3)$ to the real numbers $\log_2(3)$ and $\log_2(3) + a$ associated to the nonlattice Dirichlet polynomial in Example \ref{type1_3 example}, along with some specifications for their corresponding lattice string approximations.} \label{type1_3 poly data}
\end{table}

Choosing $m_1 = m_2 = m_3 = 0.10$, $a = 1/\sqrt{100003}$, and $\alpha_2 = \log_2(3)$, we obtain a nongeneric nonlattice Dirichlet polynomial of the form (\ref{type 1}), with rank three and with LSA constant $C = 0.002$; see Table \ref{type1_3 poly data} for data on the lattice string approximations we computed. Figure \ref{type1_3 large scale plot} shows plots of the roots of two lattice string approximations, giving a totally new large scale impression of the quasiperiodic pattern. Figure \ref{lsa_against_type1_3} shows the roots of lattice string approximations moving away from the roots of $f(s)$.
\end{example}

\section{Concluding Remarks} \label{section 5}
We close this paper by a few comments concerning future research directions motivated, in part, by the present work and by earlier work in \cite[Chapter 3]{lapidus2012fractal} (about the quasiperiodic patterns of the complex dimensions of nonlattice self-similar fractal strings), as well as by \cite{lapidus2008in}. We limit ourselves here to a few questions concerning those matters. 

While this is not a topic discussed in the present paper, it is worthwhile to mention that the LSA algorithm from \cite{lapidus2000fractal,lapidus2003complex,lapidus2006fractal,lapidus2012fractal} also establishes a deep connection to the theory of mathematical quasicrystals; for such quasicrystals, see, e.g., \cite{axel2000from}, \cite{baake2013aperiodic}, \cite{lagarias2000mathematical}, \cite{lapidus2008in}, \cite{senechal1996quasicrystals}, and the relevant references therein. See also, especially, the open problem formulated in \cite{lapidus2012fractal} (cf. \cite[Problem 3.22]{lapidus2012fractal}, restated in Problem \ref{open problem} below) asking to view the quasiperiodic set of complex dimensions of a nonlattice self-similar string (or, more generally, the set of zeros of a nonlattice Dirichlet polynomial) as a generalized mathematical quasicrystal; this problem was recently addressed by the third author in his Ph.~D.\ thesis (see Chapter 4 of \cite{voskanian2019on}), and it will be the topic of a future joint paper, \cite{lapidus2020diffraction}.

\begin{problem}\textup{(\cite[p. 86]{lapidus2006fractal}, \cite[Problem 3.22, p. 89]{lapidus2012fractal}).} \label{open problem}
Is there a natural way in which the quasiperiodic pattern of the sets of complex roots of nonlattice Dirichlet polynomials can be understood in terms of a notion of a suitable \textup(generalized\textup) quasicrystal or of an associated quasiperiodic tiling?
\end{problem}

Generalized quasicrystals of complex dimensions also play an important role in a related, but significantly more general context, in the book by the first author, \cite{lapidus2008in}.
 
Aspects of Theorem \ref{approximation theorem} (\cite[Theorem 3.18, p. 84]{lapidus2012fractal}) need to be refined. While this theorem is very helpful, in particular, for explaining when the roots of $f(s)$ start to move away from the roots of a lattice string approximation, it does not explain why they stay close near certain lines for much longer. This would suggest, among other things, that the density results and plots concerning the roots of nonlattice Dirichlet polynomials obtained in \cite[Section 3.6]{lapidus2012fractal} can be improved, at least in certain cases. 

Along similar lines, can we explain the observations made in Example \ref{type1_2 example} regarding the concentration of the roots near the extreme vertical lines $\Re(z) = a$ and $\Re(z) = b$, as well as near a third vertical line (here, the line $\Re(z) = -0.0543$); see, especially, Figure \ref{lsa_against_type1_2}. How general is this phenomenon?

Another, rather puzzling, phenomenon is the appearance in both parts of Figure \ref{type1_3 large scale plot} (about Example \ref{type1_3 example}) of very regular patterns and shapes. Can we find a good mathematical explanation for this new phenomenon? Furthermore, can we predict when related phenomena will arise and what types of shapes or patterns can be expected?

Based on the quasiperiodic patterns described in Section \ref{section 4} of the present paper, it seems reasonable to expect that the patterns for the roots of a nonlattice Dirichlet polynomial exhibit the property of {\it repetitivity} often used to qualify one of the features of mathematical quasicrystals; see, e.g., \cite{axel2000from}, \cite{baake2013aperiodic}, \cite{lagarias2000mathematical}, \cite{senechal1996quasicrystals}, \cite[Appendix F]{lapidus2008in} and the relevant references therein. It would be interesting to precisely formulate and then prove the corresponding statement in our present situation, at least for a suitable class of nonlattice Dirichlet polynomials (and, in particular, of nonlattice self-similar fractal strings). One may also ask a similar question regarding the property of {\it finite local complexity}; for the latter geometric property, see, e.g., {\it ibid}, including \cite[Definition 1.2]{lagarias2000mathematical}.

Understanding the diffraction measures and patterns of sets of roots of lattice and, eventually, of nonlattice, Dirichlet polynomials is the object of two papers in preparation by the authors (\cite{lapidus2020diffraction}, along with its sequel). For now, our main result in this context, to be discussed in \cite{lapidus2020diffraction}, concerns the lattice case. However, a conjecture of the first author, along with the LSA algorithm from \cite{lapidus2012fractal} (described in Theorem \ref{approximation theorem} above), may likely provide a clue as to what the general nonlattice case should entail. The numerical and graphical results obtained in the present work should also be helpful in this endeavor. 

\section{Acknowledgments} \label{acknowledgments}
The authors of this paper acknowledge use of the ELSA high performance computing cluster at The College of New Jersey for conducting the research reported in this paper. This cluster is funded in part by the National Science Foundation under grant numbers OAC-1826915 and OAC-1828163. The work of Michel L. Lapidus was partially supported by the National Science Foundation (under grant DMS-110775) and by the Burton Jones Endowed Chair in Pure Mathematics of which he is the Chair holder at the University of California, Riverside. We also acknowledge use of the multiprecision polynomial solver MPSolve, which we used to generate all of the plots in the current paper. 


\end{document}